\documentclass{article}

\usepackage[a4paper, total={7in, 8in}]{geometry}
\usepackage[utf8]{inputenc}   
\usepackage[T1]{fontenc}      
\usepackage{amsmath,amsthm} 
\usepackage{amssymb,mathrsfs} 
\usepackage{amssymb}
\usepackage{amsfonts}
\usepackage{graphicx}
\usepackage{enumerate}
\usepackage{graphicx} 
\usepackage{lmodern} 
\usepackage{mathtools}
\usepackage{algorithm}
\usepackage{algorithmicx,algpseudocode}
\usepackage[colorlinks]{hyperref}
\usepackage{tikz}
\usetikzlibrary{intersections}

\newtheorem{thm}{Theorem}[section]%
\newtheorem{prop}[thm]{Proposition}%
\newtheorem{lem}[thm]{Lemma}%
\newtheorem{defi}[thm]{Definition}%
\newtheorem{rem}{Remark}%

\newtheorem{algo}{Algorithm}%
\makeatletter \renewcommand{\ALG@name}{Numerical algorithm} \makeatother

\newcommand{\R}{{\mathbb R}}

\newcommand{\eps}{\varepsilon}
\newcommand{\dt}{{\Delta t}}
\newcommand{\rme}{\mathrm{e}}
\newcommand{\dps}{\displaystyle}

\newcommand{\cD}{\mathcal{D}}

\renewcommand{\leq}{\leqslant}
\renewcommand{\le}{\leqslant}
\renewcommand{\geq}{\geqslant}
\renewcommand{\ge}{\geqslant}

\title{Hybrid Monte Carlo methods for sampling probability measures on submanifolds}

\author{Tony Lelièvre$^1$, Mathias Rousset$^2$ and Gabriel Stoltz$^1$ \\
\small $1:$ Université Paris-Est, CERMICS (ENPC), Inria, F-77455 Marne-la-Vallée, France \\
\small $2:$ SIMSART team-project, Inria Rennes, France}

\date{\today}

\begin{document}

\maketitle

\begin{abstract}
  Probability measures supported on submanifolds can be sampled by adding an extra momentum variable to the state of the system, and discretizing the associated Hamiltonian dynamics with some stochastic perturbation in the extra variable. In order to avoid biases in the invariant probability measures sampled by discretizations of these stochastically perturbed Hamiltonian dynamics, a Metropolis rejection procedure can be considered. The so-obtained scheme belongs to the class of generalized Hybrid Monte Carlo (GHMC) algorithms. We show here how to generalize to GHMC a procedure suggested by Goodman, Holmes-Cerfon and Zappa for Metropolis random walks on submanifolds, where a reverse projection check is performed to enforce the reversibility of the algorithm for any timesteps and hence avoid biases in the invariant measure. We also provide a full mathematical analysis of such procedures, as well as numerical experiments demonstrating the importance of the reverse projection check on simple toy examples.
\end{abstract}

\section{Introduction and motivation}

Various applications require sampling probability measures restricted to submanifolds, for instance in molecular dynamics and computational statistics. In molecular dynamics, the typical setting corresponds to molecular systems whose configurations are distributed according to the Boltzmann--Gibbs measure, with so-called molecular constraints such as fixed bond lengths or fixed bending angles in molecules and/or fixed values of the so-called reaction coordinate function for the computation of free energy differences using thermodynamic integration. We refer {\it e.g.} to~\cite[Chapter~10]{Rap} and~\cite{Dar07,lelievre-rousset-stoltz-book-10} for applications to the computation of free energy differences, and \cite{arnold-89,leimkuhler-reich-04} for mathematical textbooks dealing with constrained Hamiltonian dynamics. In computational statistics, one prominent example is provided by Approximate Bayesian Computations~\cite{TBGD97} (see also the review article~\cite{MPRR12}), where the value of the so-called summary statistics would be exactly fixed. The motivation for sampling measures restricted to a submanifold, rather than penalizing deviations from the target value of the level set function defining the submanifold, is numerical efficiency. Other applications of sampling measures on manifolds in statistics are given in~\cite{pmlr-v22-brubaker12,DHS13}.

Sampling probability measures restricted to submanifolds is not as straightforward as sampling probability measures on the full configuration space. Dedicated methods can be used for submanifolds of a certain type, for instance algebraic manifolds defined as the zeros of polynomial functions~\cite{BM18}. We consider here submanifolds defined as level sets of a generic nonlinear function of the coordinates of the system. Although it is possible to sample probability measures restricted to such submanifolds using diffusion processes, the discretization of these processes with a timestep $\Delta t > 0$ induces Markov chains whose invariant measure departs from the target measure, with errors of order~$\Delta t^\alpha$, where $\alpha >0$ depends on the weak order of the time-discretization scheme (see for instance~\cite{faou-lelievre-09}). One way to remove the bias in the discretization of stochastic differential equations, and possibly to stabilize the numerical scheme in order to allow for larger timesteps, is to rely on a Metropolis--Hastings procedure~\cite{MRRTT53,Hastings70}, considering the outcome of the numerical discretization as a proposal to be accepted or rejected. This requires however to evaluate the Metropolis ratio, in particular the probability to come back to the initial configuration starting from the one proposed by the numerical scheme. For probability measures defined on submanifolds, the proposition kernel is however typically not explicit, since the proposal move usually consists of an unconstrained step followed by a projection procedure to the submanifold of interest. The latter part of the proposed move prevents writing down a simple and analytical formula for the likelihood of a given proposal move. A more convenient strategy is to look for proposal moves which are symmetric, so that the Metropolis ratio reduces to the ratio of the probabilities of the initial and proposed states. Hybrid Monte Carlo\footnote{This scheme is now named ``Hamiltonian Monte Carlo'' in the computational statistics literature.} (HMC) provides such proposals, which are obtained from the composition of one step of integrators of constrained Hamiltonian dynamics, which are reversible up to momentum reversal, such as RATTLE~\cite{andersen-83} (this property is proved for sufficiently small timesteps, for example in the monographs~\cite{hairer-lubich-wanner-06,leimkuhler-reich-04}). It is in fact possible to resort to generalized HMC (GHMC)~\cite{horowitz-91} and other variants of HMC based on partial resampling of momenta (see~\cite{lelievre-rousset-stoltz-12} and~\cite[Section~3.3.5.4]{lelievre-rousset-stoltz-book-10}), geodesic integrators~\cite{LM16}, the use of non-quadratic kinetic energies~\cite{ST18}, etc.  

Goodman, Holmes-Cerfon and Zappa recently showed, by a clever geometric construction, how to construct reversible Metropolis random walks on submanifolds~\cite{zappa-holmes-cerfon-goodman-17}. In fact, their construction can be interpreted as a one-step HMC scheme where the gradient of the potential energy in the proposal move is zero. Moreover, they also carefully discuss that, in addition to performing some Metropolis correction based on a probability ratio, it should be checked that the proposal mechanism is actually reversible, in the sense that it is possible to come back to the initial state starting from the proposed one. This additional check is usually not performed in the algorithms used in molecular dynamics since it is implicitly assumed that timesteps are sufficiently small for the integrators of the constrained Hamiltonian dynamics to be reversible up to momentum reversal. When the timesteps are not sufficiently small, performing this additional reverse check is however necessary to obtain unbiased samples, as already demonstrated in~\cite{zappa-holmes-cerfon-goodman-17}.

The aim of our work is first to build GHMC-like algorithms with reverse check to sample without bias probability distributions on submanifolds. In particular, we extend the algorithm proposed in~\cite{zappa-holmes-cerfon-goodman-17} by allowing for non-zero gradient forces in the proposal move. For example, we construct an unbiased Metropolis-adjusted Langevin algorithm (MALA)~\cite{RDF78,RT96} to sample without bias probability distributions on submanifolds. Second, we provide a complete mathematical analysis of the need and impact of the reversibility check in GHMC-like schemes. The main difference with the approaches in~\cite{lelievre-rousset-stoltz-book-10,lelievre-rousset-stoltz-12} is that the algorithm samples the full constrained phase-space measure associated with the target measure on the submanifold, and not the phase-space measure restricted to configurations where the projection step in the RATTLE scheme is possible. In essence, the algorithm consists in rejecting moves for which the projection cannot be enforced or RATTLE is not reversible up to momentum reversal. See Remarks~\ref{rmk:restricted_vs_constrained} and~\ref{rem:truncated_momenta} for a more precise discussion on the difference between these approaches. 

The reversibility check turns out to be crucial for numerical efficiency, at least for the numerical simulations reported in this work: the timesteps for which the numerical method is the most efficient lead to situations where a substantial fraction of the proposed moves violates the reversibility property of the underlying RATTLE scheme. Not rejecting them properly leads to large biases.

\medskip

This article is organized as follows. We present a mathematical framework for our analysis in Section~\ref{sec:mathematical_setting}. The key ingredient is to understand the reversibility properties arising from the Metropolization of the Hamiltonian dynamics, see Section~\ref{sec:RATTLE}. We illustrate our analysis on a simple numerical example in Section~\ref{sec:numerics}, where we also summarize the complete algorithm in pseudo-code in Section~\ref{sec:summary_algo}.

\section{Mathematical setting and results}
\label{sec:mathematical_setting}

We first describe in Section~\ref{sec:geometry} the target probability measures we are interested in sampling. The core of our analysis is presented in Section~\ref{sec:RATTLE}, where we show how to rigorously formalize the reversibility properties of discretizations of the Hamiltonian dynamics on a submanifold for possibly large timesteps. One crucial ingredient in this analysis is the Lagrange multiplier function, which associates to a state on the submanifold and another state close to the submanifold (obtained by an unconstrained step of the dynamics) a new state on the submanifold. Examples of such Lagrange multiplier functions are provided in Section~\ref{sec:example}. Once the discretization of the Hamiltonian part of the dynamics is clear, GHMC schemes follow in a straightforward way; see Section~\ref{sec:GHMC_schemes}.

\subsection{Geometric setting}
\label{sec:geometry}

Let us first introduce the measures we are interested in sampling, namely the target measure~\eqref{eq:nu} below, as well as the probability measure~\eqref{eq:mu} on an extended space, which admits~\eqref{eq:nu} as a marginal. We only provide the most essential objects needed for our analysis. For more details and motivations for the definitions and results given here, we refer the reader to the standard reference textbooks~\cite{abraham-marsden-78,arnold-89}, and also to~\cite[Chapter 3.3.2]{lelievre-rousset-stoltz-book-10} for a self-contained presentation.

\subsubsection{Target measure}

We denote by~$d$ the dimension of the ambient space, and consider a submanifold $\mathcal{M} \subset \mathbb{R}^d$ defined as the zero level set of a given smooth function $\xi:\R^d \to \R^m$ (with $m<d$):
\[
\mathcal{M} = \Big\{q \in \R^d, \, \xi(q) = 0 \Big\}.
\]
For example, the function~$\xi$ encodes molecular constraints or reaction coordinates in molecular dynamics, or is a summary statistics in Approximate Bayesian Computations. Let $M \in \R^{d \times d}$ be a fixed symmetric positive definite matrix, which is interpreted as the mass tensor below. The ambient space $\R^d$ is endowed with the scalar product $\langle q, \tilde q \rangle_M=q^T M \tilde q$. One could take for simplicity $M=\mathrm{Id}$ but it is sometimes useful to consider non identity mass matrices for numerical purposes~\cite{girolami-calderhead-11}. The function $\xi$ is assumed to be smooth on a neighborhood of ${\mathcal M}$ in~$\R^d$ and such that 
\begin{equation}
  \label{eq:GM}
  G_M(q) = \left[\nabla \xi(q)\right]^T M^{-1} \nabla \xi(q) \in \mathbb{R}^{m \times m}
\end{equation} 
is an invertible matrix for all $q$ in this neighborhood. In the latter expression, $\nabla \xi \in \mathbb{R}^{d \times m}$ is a rectangular matrix whose columns $\nabla \xi_i$ are the gradients of the components of $\xi=(\xi_1,\dots,\xi_m)$. The submanifold $\mathcal M$ of $\R^d$ is thus a smooth submanifold of co-dimension $m$. In the following, $\sigma^M_{\mathcal M}(dq)$ denotes the Riemannian measure on ${\mathcal M}$ induced by the scalar product $\langle \cdot, \cdot \rangle_M$ defined in the ambient space $\R^d$. Finally, let us introduce a smooth function $V: \R^d \to \R$. The aim of the algorithms presented below is to sample the
probability measure
\begin{equation}
  \label{eq:nu}
  \nu(dq)=Z_\nu^{-1} \, \rme^{- V(q)} \, \sigma^M_{\mathcal M}(dq), \qquad Z_\nu = \int_{\mathcal M}  \rme^{-V(q)} \, \sigma^M_{\mathcal M}(dq),
\end{equation}
where $Z_\nu$ is assumed to be finite.

\subsubsection{Extended measure in phase-space}

In order to construct numerical methods based on a Metropolis--Hastings procedure to sample~\eqref{eq:nu}, it is convenient to introduce algorithms in a higher dimensional space. The configuration of the system is now described by $(q,p)$, where, for a given position $q \in {\mathcal M}$, momenta $p$ belong to~$T^*_q {\mathcal M}$, the cotangent space to~$\mathcal{M}$ at~$q$. This cotangent space can be identified with a linear subspace of~$\R^d$:
\[
T^*_q \mathcal M = \Big\{ p \in \R^d, \, \left[\nabla \xi(q)\right]^T M^{-1} p = 0 \Big\} \subset \R^d.
\]
Let us denote by
\[
T^* \mathcal M = \Big\{ (q,p) \in \mathbb{R}^{d} \times \mathbb{R}^{d}, \, \xi(q)=0 \text{ and } \left[\nabla \xi(q)\right]^T M^{-1} p = 0 \Big\} \subset \R^d \times \R^d
\]
the associated cotangent bundle, which can be seen as a submanifold of $\R^d \times \R^d$. The phase space Liouville measure on $T^* \mathcal M$ is denoted by $\sigma_{T^* \mathcal M }(dq \, dp)$. It can be written in tensorial form:
\begin{equation}
  \label{eq:tensor}
  \sigma_{T^*\mathcal M }(dq \, dp)= \sigma^{M}_{\mathcal M}(dq) \,\sigma^{M^{-1}}_{T^*_q \mathcal M}(dp),
\end{equation}
where $\sigma^{M^{-1}}_{T^*_q \mathcal M}(dp)$ denotes the Lebesgue measure on $T^*_q \mathcal M$ induced by the scalar product $\langle p,\tilde p \rangle_{M^{-1}}=p^T M^{-1} \tilde p$ in $\R^d$. The measure $\sigma_{T^* \mathcal M }(dq \, dp)$ does not depend on the choice of the mass tensor~$M$ (see for example~\cite[Proposition~3.40]{lelievre-rousset-stoltz-book-10}).

In order to relate~\eqref{eq:tensor} with the target measure~\eqref{eq:nu}, we introduce the Hamiltonian function
\[
H(q,p)=V(q) + \frac{p^T M^{-1} p}{2}.
\]
The phase-space measure we would like to sample is
\begin{equation}
  \label{eq:mu}
  \mu(dq \,dp) = Z_\mu^{-1}\rme^{-H(q,p)} \, \sigma_{T^*\mathcal M }(dq \, dp),
  \qquad
  Z_\mu=\int_{T^*\mathcal M} \rme^{-H(q,p)} \, \sigma_{T^*\mathcal M }(dq \, dp),
\end{equation}
where $Z_\mu$ is finite when $Z_\nu$ is. Note that, thanks to the tensorization property~\eqref{eq:tensor} and the separability of the Hamiltonian function,
\begin{equation}
\label{eq:mu_tensor}
\mu(dq \, dp) = \nu(dq) \, \kappa_q(dp),
\end{equation}
where $\nu$ is defined in~\eqref{eq:nu} and $\kappa_q$ is a Gaussian measure on~$T^*_q \mathcal M$:
\begin{equation}
  \label{eq:kappa}
  \kappa_q(dp)=(2\pi)^{\frac{m-d}{2}} \exp \left(-\frac{p^T M^{-1} p}{2}\right)\sigma^{M^{-1}}_{T^*_q \mathcal M}(dp).
\end{equation}
In particular, the marginal of $\mu$ in the variable~$q$ is exactly~$\nu$.

\begin{rem}
  \label{rmk:restricted_vs_constrained}
  As will be made clear in Section~\ref{sec:GHMC_schemes}, the target measure sampled by the algorithms presented in this work is the phase-space measure~\eqref{eq:mu}. In previous works~\cite{lelievre-rousset-stoltz-book-10,lelievre-rousset-stoltz-12}, the target measure was a restricted phase-space measure, of the form $\mu(dq \, dp) \mathbf{1}_{D_{\Delta t}}(q,p)$, where $D_{\Delta t}$ is a subset of $T^*\mathcal{M}$ containing phase-space configurations for which the projection in the RATTLE algorithm (see~\eqref{eq:RATTLE} below) is possible. The marginal of the restricted measure in the variable~$q$ is not~$\nu$ in general. This is however the case, for example, when $D_{\Delta t}(q,p) = \{ (q,p)\in T^*\mathcal{M}, \ |p|^2 \leq R_{\Delta t} \}$ for some $R_{\Delta t} > 0$. In essence, the aim of the analysis provided in the sequel of this article is to make the set~$D_{\Delta t}$ as large as possible (see the definition~\eqref{eq:B}); and to sample the full phase-space measure rather than its restricted counterpart by taking into account configurations outside this maximal set. 
\end{rem}

\subsection{The RATTLE dynamics with reverse projection check}
\label{sec:RATTLE}

This section, which is the core of our analysis, shows how to construct a reversible map based on the RATTLE integrator (recalled in Section~\ref{sec:std_RATTLE}) which is well defined for arbitrary large timesteps. We discuss in Section~\ref{sec:proj} which properties the Lagrange multipliers used to enforce the position constraint should enjoy in order to construct globally well defined integrators (see Section~\ref{sec:Psi_dt}). This  finally allows constructing globally well defined reversible integrators in Section~\ref{sec:psi_B}.

\subsubsection{The RATTLE integrator}
\label{sec:std_RATTLE}

In order to sample the measures $\nu$ and $\mu$, we need as a crucial ingredient the RATTLE integrator~\cite{andersen-83}, which is a velocity-Verlet algorithm for Hamiltonian dynamics with constraints; see~\cite{hairer-lubich-wanner-06,leimkuhler-reich-04}. For a given timestep $\Delta t>0$ and a configuration $(q^n,p^n) \in T^* {\mathcal M}$, one step of the RATTLE algorithm proceeds as follows:
\begin{equation}
  \label{eq:RATTLE}
  \left\{ \begin{aligned} 
  & p^{n+1/2} =  p^{n} - \frac{\Delta t}{2} \nabla V (q^{n}) + \nabla \xi(q^n) \, \lambda^{n+1/2}, &\\
  & q^{n+1} =  q^{n} + \Delta t \, M^{-1} \, p^{n+1/2}, &\\
  & \xi(q^{n+1}) = 0, \qquad &(C_{q}) \\
  & p^{n+1} =  p^{n+1/2} - \frac{\Delta t}{2} \nabla V (q^{n+1}) + \nabla \xi(q^{n+1}) \, \lambda^{n+1}, &\\
  &\left[\nabla\xi (q^{n+1})\right]^T M^{-1} p^{n+1} = 0, \qquad &(C_{p})
\end{aligned} \right.
\end{equation}
where $\lambda^{n+1/2} \in \R^m$ are the Lagrange multipliers associated with the position constraints $(C_{q})$, and $\lambda^{n+1}\in \R^m$ are the Lagrange multipliers associated with the velocity constraints $(C_{p})$. The RATTLE scheme is a second order discretization of the constrained Hamiltonian dynamics
\[
\left\{
\begin{aligned}
  dq_t& = M^{-1} p_t \, dt \, ,\\
  dp_t&= -\nabla V(q_t) \, dt + \nabla \xi(q_t) \, d \lambda_t \, ,\\
  \xi(q_t)&=0 \,.
\end{aligned}
\right.
\]
 
In~\eqref{eq:RATTLE}, the momentum projection~$(C_p)$ is always well-defined since it corresponds to a linear projection on the vector space $T^*_{q^{n+1}} \mathcal M$. In fact,
\[
\lambda^{n+1} = -G_M\left(q^{n+1}\right)^{-1} \left[\nabla\xi (q^{n+1})\right]^T M^{-1}\left( p^{n+1/2} - \frac{\Delta t}{2} \nabla V (q^{n+1}) \right),
\]
and thus
\[
p^{n+1} = \Pi_{T^*_{q^{n+1}} \mathcal M}\left( p^{n+1/2} - \frac{\Delta t}{2} \nabla V (q^{n+1}) \right), 
\]
where 
\begin{equation}
  \label{eq:proj_T*M}
  \Pi_{T^*_{q} \mathcal M} = {\rm Id} - \nabla \xi(q) G_M^{-1}(q) [\nabla \xi(q)]^T M^{-1}
\end{equation}
is the orthogonal projection onto $T^*_{q} \mathcal M$, the orthogonality being with respect to the scalar product $\langle\cdot,\cdot\rangle_{M^{-1}}$. 

On the contrary, there may be many solutions or no solution at all to the nonlinear projection step used to enforce the position constraint~$(C_{q})$.
Note that, from~\eqref{eq:RATTLE}, 
\begin{equation}
  \label{eq:rewrite_increment}
  q^{n+1} = \tilde{q}^n + \Delta t \, M^{-1} \nabla \xi(q^n) \lambda^{n+1/2} \text{  where  } \tilde{q}^n = q^n + \dt M^{-1} \left(p^n - \frac{\dt}{2} \nabla V(q^n)\right).
\end{equation}
This expression suggests that the Lagrange multipliers can be thought of as functions of the current position~$q^n$ (which provides the direction $M^{-1} \nabla \xi(q^n)$ along which the projection is performed) and of the position $\tilde{q}^n$ obtained after one step of the unconstrained dynamics. Actually, it turns out to be more convenient to consider $\dt \lambda^{n+1/2}$ as the Lagrange multiplier to define the projection, in order to analyze the properties of the algorithm as $\Delta t$ varies.  It is the aim of the next section to precisely define the requirements needed on the Lagrange multiplier function $\Lambda(q^n,\tilde{q}^n)$ which is used to define $\dt \lambda^{n+1/2}=\Lambda(q^n,\tilde{q}^n)$. This is indeed a crucial point since there may be more than one possible projected position onto the submanifold, or maybe none (see Figure~\ref{fig:illustration_non_reversibility} below). We will then give in Section~\ref{sec:example} examples of Lagrange multiplier functions satisfying these requirements.

\subsubsection{The Lagrange multiplier function}
\label{sec:proj}

In order to make precise the properties we need on the numerical scheme, let us introduce what we call the Lagrange multiplier function, which encodes the numerical procedure used in~\eqref{eq:RATTLE} in order to enforce the position constraint~$(C_{q})$. This function depends on two parameters: the position~$q \in \mathcal M$ which provides the direction $M^{-1}\nabla \xi(q)$ along which the projection is sought and the position $\tilde{q} \in \mathbb{R}^d$ to be projected onto the submanifold~$\mathcal{M}$. As hinted at at the end of the previous section, the projection procedure is not well-defined in general: there can be no such projection, or multiple ones. The situation depends on the number of intersections of the $m$-dimensional vector space $\{ \tilde{q} + M^{-1}\nabla \xi(q)\theta, \theta \in \mathbb{R}^m\}$ with~$\mathcal{M}$. For instance, in the situation considered in Figure~\ref{fig:illustration_non_reversibility}, there are two possible projections~$q''$ and~$q$ in the direction $M^{-1}\nabla \xi(q)$ when starting from~$\widehat{q}$.

In order to ensure the well-posedness of the projection, we require various properties on the Lagrange multiplier function, made precise in Definition~\ref{def:proj_sel} below. The first property is that there exists indeed a projection. Moreover, in cases when there are several possible projection, the $C^1$ regularity of the Lagrange multiplier function makes sure that the solution is chosen on the same branch, at least locally. Let us emphasize that an important ingredient in the definition of the Lagrange multiplier function is its domain~$\mathcal{D}$.

\begin{defi}
  \label{def:proj_sel}
  An admissible Lagrange multiplier function $\Lambda$ is a $C^1$ function defined on an open set $\mathcal{D}$ of $\mathcal{M} \times \R^d$ with values in $\R^m$, and satisfying the following two properties: 
  \begin{itemize} 
  \item the projection property:
    \begin{equation}
      \label{eq:Pi}
      \forall (q,\tilde{q}) \in \cD, \qquad  \tilde{q}+ M^{-1} \nabla \xi(q) \Lambda(q,\tilde{q}) \in \mathcal M;
    \end{equation}
  \item the non-tangential projection property:
    \begin{equation}
      \label{eq:non_tang}
      \forall (q,\tilde{q}) \in \cD, \qquad \left[\nabla \xi\left(\tilde{q} + M^{-1}\nabla \xi(q) \Lambda(q,\tilde{q})\right)\right]^T M^{-1}\nabla \xi(q) \in \R^{m  \times m} \text{\, is invertible.}
    \end{equation}
  \end{itemize}
\end{defi}

In the following, we will always consider sets $\mathcal{D}$ which contain the pairs $(q,\tilde{q}) \in \mathcal{M} \times \mathcal{M}$ such that the matrix $\left[\nabla \xi\left(\tilde{q}\right)\right]^T M^{-1}\nabla \xi(q)$ is invertible, in which case $\Lambda(q,\tilde{q})=0$. Examples of admissible Lagrange multiplier functions are given in Section~\ref{sec:example}. We have in mind two typical examples. The first idea is to consider the projection of $\tilde{q}$ on $\mathcal M$ along the direction $M^{-1} \nabla \xi(q)$ which selects the Lagrange multiplier with the smallest norm (projection to the ``closest point''). This is formalized in Section~\ref{sec:ex1_} using the implicit function theorem. The second idea, which is more practical, is to use a Newton's algorithm to enforce the constraint (see for instance the discussion in~\cite{leimkuhler-reich-04} and Section~\ref{sec:ex2} below). 

The projection~\eqref{eq:Pi} is consistent with the projection needed in the RATTLE integrator to satisfy the position constraint. More precisely, the Lagrange multiplier in~\eqref{eq:RATTLE} writes
\[
\lambda^{n+1/2} = \frac{1}{\dt}\Lambda\left(q^n,q^n + \dt M^{-1} \left[p^n - \frac{\dt}{2} \nabla V(q^n)\right]\right).
\]
The set $\cD$ and the Lagrange multiplier function~$\Lambda$ encode the numerical procedure used to compute the Lagrange multiplier associated with the position constraints. For $(q,\tilde{q}) \in \mathcal{M} \times \R^d$, the Lagrange multiplier function selects some configuration $q'$ in the intersection of $\mathcal M$ and the affine space $\tilde{q} + \mathrm{Span}\{\delta_1,\ldots, \delta_m\}$ containing the position $\tilde{q}$ obtained after one step of the unconstrained dynamics, and generated by the linearly independent vectors $\delta = (\delta_1, \ldots, \delta_m) = M^{-1}\nabla \xi(q)$ defining the vector space orthogonal to the cotangent space at~$q$ for the scalar product $\langle\cdot,\cdot\rangle_{M^{-1}}$. See Figure~\ref{fig:illustration_non_reversibility} below for an illustration.

Let us emphasize that we assume that the function~$\Lambda$ is $C^1$. This requires discarding neighborhoods of initial positions~$q$ and unconstrained moves~$\tilde{q}$ for which the obtained projected positions change abruptly; see Figure~\ref{fig:illustration_non_C1} below for an illustration.

The non-tangential condition~\eqref{eq:non_tang} means that, for the selected projection $q' = \tilde{q} + M^{-1} \nabla \xi(q) \Lambda(q,\tilde{q})\in \mathcal M$, any basis of the cotangent space $T_{q'}^* \mathcal M$ is linearly independent of the directions $\nabla \xi(q) = (\nabla \xi_1(q), \ldots, \nabla \xi_m(q))$ used to perform the projection; see Figure~\ref{fig:non_tang} for an illustration of a situation where~\eqref{eq:non_tang} is not satisfied. When the non-tangential condition is not satisfied, there may be infinitely many ways to project the position back to the submanifold. The non-tangential condition ensures that only countably many projections have to be considered, and that the projection is locally unique.  
The non-tangential condition~\eqref{eq:non_tang} also has several other motivations:
\begin{itemize}
\item on the theoretical side, it allows us to ensure that RATTLE is reversible up to momentum reversal on an \emph{open set} constructed from~$\mathcal{D}$; see the informal discussion before the proof of Lemma~\ref{lem:B}. The non-tangential condition is a sufficient condition in this context.
\item a more practical remark is that the set of configurations~$(q,\tilde{q})$ on which~\eqref{eq:non_tang} does not hold usually is a set of measure~0 with respect to the Lebesgue measure~$dq \, d\tilde{q}$. It is therefore not restrictive to assume that it holds.
\item from a numerical viewpoint, the Lagrange multiplier is often computed using a Newton algorithm (see Section~\ref{sec:ex2}). In this case, the numerical procedure is well posed only if the non-tangential condition holds at the limiting point. The latter condition is therefore necessary in this context to run the numerical method. 
\end{itemize}

\begin{figure}
  \begin{center}
    \begin{tikzpicture}
      \draw [fill] (6,4.5) circle [radius=0.05] node [black,right=1] {$q$}; 
      \draw [fill] (6,2) circle [radius=0.05] node [black,below=2] {$\tilde{q}$};
      \draw [->] (6,4.5) -- (4.5,4.5);
      \draw (5,4.5) node [black,above=1] {$\nabla \xi(q)$};
      \draw (6,4.5) -- (6,2);
      \draw [thick] (-4,-0.5) 
      to [out=80,in=180] (2,2)
      to [out=0,in=270] (6,4.5)
      to [out=90,in=0] (4,5.5);
      \draw [dashed] (-2,2) 
      to [out=0,in=180] (6,2);
      \draw [->] (2,2) -- (2,4);
      \draw [fill] (2,2) circle [radius=0.05] node [black,below=2] {$q' = \tilde{q}+ M^{-1}\nabla \xi(q) \Lambda(q,\tilde{q})$};
      \draw (2,4) node [black,above=1] {$\nabla \xi(q')$};
    \end{tikzpicture}
    \caption{\label{fig:non_tang} Illustration of a situation when the non-tangential condition~\eqref{eq:non_tang} is not satisfied, for $m=1$ and $M=\mathrm{Id}$.}
  \end{center}
\end{figure}
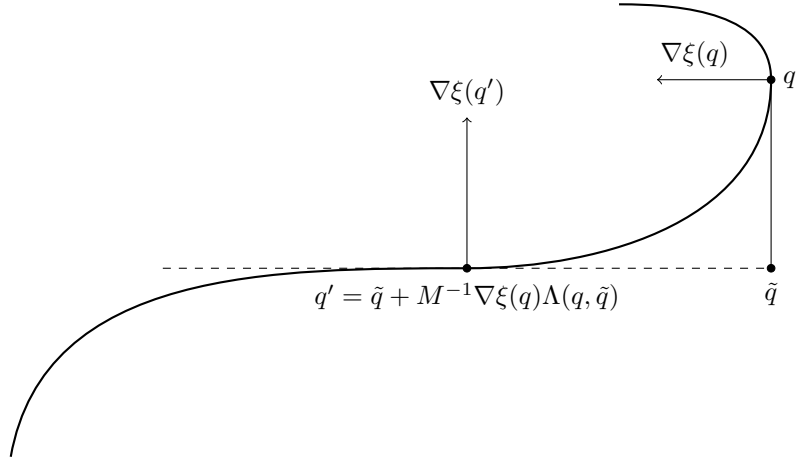

\subsubsection{The map $\Psi_{\Delta t}$}
\label{sec:Psi_dt}
 
The numerical integrator $\Psi_{\Delta t}$ we consider in the following is built as one step of the RATTLE scheme using the Lagrange multiplier function~$\Lambda$, composed with momentum reversal. The reason why we compose with momentum reversal is that we want the integrator to be an involution. This leads to a crucial simplification in the Metropolis ratio for HMC-like schemes, see Section~\ref{sec:GHMC_schemes}. 

More precisely, for a given timestep $\Delta t$, let us introduce the {\em admissible set} $A \subset T^* \mathcal M$ as the configurations~$(q,p)$ for which the positions obtained after the unconstrained integration of the positions in~\eqref{eq:RATTLE} belongs to~$\cD$ (and hence can be projected back onto the submanifold $\mathcal M$ using~$\Lambda$):
\begin{equation}
  \label{eq:def_A}
  A = \left\{ (q,p) \in T^*\mathcal{M}, \, \ \left(q, q+ \Delta t \, M^{-1} \left[p - \frac{\Delta t}{2} \nabla V (q) \right] \right) \in \cD  \right\}.
\end{equation}
This set is the natural counterpart of the domain~$\mathcal{D}$ of the Lagrange multiplier function~$\Lambda$. In essence, it is the set of phase-space configurations which can be integrated with one step of the RATTLE algorithm, and for which the projected positions vary smoothly with respect to~$(q,p)$. Here again, as made precise later on in Lemma~\ref{lem:B}, the union of sets such as~\eqref{eq:def_A} will allow to decompose the configuration space into patches corresponding to a certain choice of the projection. 

By continuity of~$V$, $A$ is an open set of $T^*\mathcal{M}$. We present in Section~\ref{sec:ex1_} explicit examples where we can prove that this set is not empty, see Lemma~\ref{lem:A_B_non_empty}. For a given $(q,p) \in A$, we define
\begin{equation}
  \label{eq:Psi}
  \Psi_{\Delta t}(q,p)=(q^{1},-p^{1}),
\end{equation} 
where $(q^{1},p^{1})$ is defined by one step of the RATTLE dynamics (see~\eqref{eq:RATTLE}):
\begin{equation}
  \label{eq:RATTLE_num}
  \left \{ \begin{aligned} 
    & \lambda^{1/2} = \frac{1}{\dt}\Lambda\left(q,q+\dt M^{-1} \left[p - \frac{\dt}{2} \nabla V(q)\right]\right), \\
    & p^{1/2} =  p - \frac{\Delta t}{2} \nabla V (q) + \nabla \xi(q) \, \lambda^{1/2}, \\
    & q^{1} = q + \Delta t \, M^{-1} \, p^{1/2}, \\
    & p^{1} = p^{1/2} - \frac{\Delta t}{2} \nabla V (q^{1}) + \nabla \xi(q^1) \lambda^1, \\
    & \left[\nabla \xi(q^1)\right]^T M^{-1} p^1 = 0. 
\end{aligned} \right.
\end{equation}
Notice that, by construction, $q^1 \in \mathcal M$ satisfies
\[
q^1 = \tilde{q}+ M^{-1}\nabla \xi(q) \Lambda(q,\tilde{q}) \text{  where  } \tilde{q} = q+ \Delta t \, M^{-1} \left(p - \frac{\Delta t}{2} \nabla V (q) \right).
\]
The non-tangential property~\eqref{eq:non_tang} implies that 
\begin{equation}
  \label{eq:non_tang_q^1}
\text{for any } (q,p) \in A, \, \left[\nabla \xi(q^1)\right]^T M^{-1} \nabla \xi(q) \text{ is invertible, where } q^1\text{ is defined by~\eqref{eq:RATTLE_num}}.
\end{equation}
Recall also that the last equations in~\eqref{eq:RATTLE_num} could be rephrased more explicitly as 
\[
p^{1} = \Pi_{T^*_{q^1} \mathcal M}\left( p^{1/2} - \frac{\Delta t}{2} \nabla V (q^{1}) \right),
\]
but we keep the above formulation since it is useful for the proof of Lemma~\ref{lem:B} below.

\medskip

Let us now state the crucial property of the map $\Psi_{\Delta t}$ we will need in the following.

\begin{prop}
  \label{prop:psi_Dt}
  Assume that $\Lambda: \cD  \to \R^m$ is an admissible Lagrange multiplier function in the sense of Definition~\ref{def:proj_sel}, and recall the definition~\eqref{eq:def_A} of the open set $A$. Then, the application $\Psi_{\Delta t}:A \to T^*{\mathcal M}$ defined by~\eqref{eq:Psi}-\eqref{eq:RATTLE_num} is a $C^1$ local diffeomorphism, locally preserving the phase-space measure $\sigma_{T^* {\mathcal M}} (dq \, dp)$.
\end{prop}

\begin{proof}
By assumption, $\Lambda$, $\nabla \xi$ and $\nabla V$ are $C^1$ on $A$, so that the RATTLE numerical flow $\Phi_{\Delta t}: (q,p) \mapsto (q^{1},p^{1})$ is also $C^1$. The symplecticity of $\Phi_{\Delta t}$ can be checked as in~\cite{leimkuhler-reich-94} and~\cite[Section~VII.1.3]{hairer-lubich-wanner-06} by computing $\nabla \Phi_{\Delta t}$. Since $\Phi_{\Delta t}$ is a smooth symplectic map, it is a measure-preserving local diffeomorphism. Note next that $\Psi_{\Delta t}$ is obtained by composing $\Phi_{\Delta t}$ with the momentum reversal $(q,p) \mapsto (q,-p)$. Therefore, $\Psi_{\Delta t}$ is also a $C^1$ local diffeomorphism. Since the momentum reversal preserves the measure $\sigma_{T^* {\mathcal M}} (dq \, dp)$, we can conclude that $\Psi_{\Delta t}$ also preserves the latter measure.
\end{proof}

\subsubsection{The map $\Psi_{\Delta t}^{\rm rev}$}
\label{sec:psi_B}

The RATTLE dynamics with momentum reversal and reverse projection check is now defined as follows: for any $(q,p) \in T^*{\mathcal M}$,
\begin{equation}
  \label{eq:PsiB}
  \Psi_{\Delta t}^{\rm rev}(q,p)=\Psi_{\Delta t}(q,p) 1_{\{(q,p) \in B\}} + (q,p) 1_{\{(q,p) \not\in B\}},
\end{equation}
where the set $B \subset A \subset T^*\mathcal{M}$ is
\begin{equation}
  \label{eq:B}
  B = \Big\{ (q,p) \in A, \, \Psi_{\Delta t}(q,p) \in A \text{ and } (\Psi_{\Delta t} \circ \Psi_{\Delta t})(q,p) = (q,p) \Big\}.
\end{equation}
In words, the set~$B$ is the set of configurations~$(q,p)$ for which the RATTLE scheme is well posed and its ouput~$(q',p')$ depends continously on the input configuration~$(q,p)$, and moreover the RATTLE scheme is also well defined for the input~$(q',-p')$. This will allow to check whether an application of RATTLE to~$(q',-p')$ will lead back to~$(q,-p)$, the initial configuration up to momentum reversal. The reversibility of RATTLE can therefore be true only in a subset of~$B$, which shows the importance of understanding the properties of this ensemble.

We provide in Sections~\ref{sec:ex1_}-\ref{sec:rattle_imp} an explicit example where the set $B$ can be proved to be non empty (see Lemma~\ref{lem:A_B_non_empty}). The motivation for the definition~\eqref{eq:PsiB} of $\Psi_{\Delta t}^{\rm rev}$ is to construct a piecewise measure-preserving diffeomorphism defined on the whole space $T^*\mathcal M$, and not just on~$B$.

More concretely, $\Psi_{\Delta t}^{\rm rev}(q,p)$ is obtained from $(q,p) \in T^*{\mathcal M}$ by the following procedure:
\begin{enumerate}[(1)]
\item check if $(q,p)$ is in $A$; if not return $(q,p)$;
\item when $(q,p) \in A$, compute the configuration $(q^{1},p^{1})$ obtained by one step of the RATTLE scheme~\eqref{eq:RATTLE_num}; 
\item check if $(q^1,-p^1)$ is in $A$; if not, return $(q,p)$;
\item compute the configuration $(q^2,-p^2)$ obtained by one step of the RATTLE scheme~\eqref{eq:RATTLE_num} starting from $(q^{1},-p^{1})$;
\item if $(q^2,p^2)=(q,p)$, return $(q^{1},-p^{1})$; otherwise return $(q,p)$.
\end{enumerate}
The steps (3)-(4)-(5) correspond to the {\em reverse projection check}. For sufficiently small timesteps, it is expected that these steps are useless since the RATTLE integrator is known to be reversible with respect to momentum reversal (see the definitions in~\cite[Section~V.1]{hairer-lubich-wanner-06}): denoting by~$S$ the momentum reversal operator acting as $S(q,p) = (q,-p)$, and by $\Phi_{\Delta t}$ the flow of the RATTLE scheme, it holds $S \circ \Phi_{\Delta t} = \Phi_{\Delta t}^{-1}\circ S$. In fact, this property, named $S$-reversibility, is equivalent to the symmetry and time-reversibility of the method; see~\cite[Theorem~V.1.5]{hairer-lubich-wanner-06}. This is formalized in Lemma~\ref{lem:rev_Psi_dt} below. However, in practice, ensuring that the timestep is sufficiently small so that the $S$-reversibility is satisfied is not obvious, and it may actually be useful to use large timesteps, in which case the  {\em reverse projection check} is necessary, see Figure~\ref{fig:illustration_non_reversibility} for an illustration. We refer to Section~\ref{sec:summary_algo} for a discussion on how to check the conditions $(q,p) \in A$ and $(q^1,-p^1) \in A$.

\begin{figure}
  \begin{center}
    \begin{tikzpicture}
      \draw [thick] (0,0) 
      to [out=270,in=180] (2,-2)
      to [out=0,in=180] (4,-1)
      to [out=0,in=100] (6,-2);
      \draw [thick] (0,0) 
      to [out=90,in=180] (2,2)
      to [out=0,in=180] (4,1)
      to [out=0,in=240] (7,3);
      \draw [fill] (5.65,-1.4) circle [radius=0.05] node [black,below=2] {$q$}; 
      \draw [dashed] (5.65,-1.4) to (-0.5,2.5);
      \draw [->] (5.65,-1.4) to (4,-0.35);
      \draw (4,-0.35) node [black,right=3] {$p$};
      \draw [->] (5.65,-1.4) -- (6.5,-0.06);
      \draw (6.5,-0.06) node [black,above=1] {$\nabla \xi(q)$};
      \draw [fill] (2,2) circle [radius=0.05] node [black,above=2] {$q' = \tilde{q} + M^{-1} \nabla \xi(q) \Lambda(q,\tilde{q})$}; 
      \draw [->] (2,2) -- (3.5,2);
      \draw (3.5,2) node [black,below=1] {$-p'$};
      \draw [dashed] (2,2) to (6.5,2);
      \draw (5.65,2) to (5.65,-1.4);
      \draw [fill] (5.65,1.46) circle [radius=0.05] node [black,right=1] {$q''$};
      \draw [->] (2,2) -- (2,0);
      \draw (2,0) node [black,below=1] {$\nabla \xi(q')$}; 
      \draw (1.5,1.215) to (2,2); 
      \draw [fill] (1.5,1.215) circle [radius=0.05] node {};
      \draw (1.1,0.9) node [black] {$\tilde{q} = q + v$};
      \draw [fill] (5.65,2) circle [radius=0.05] node [black,above=1] {$\widehat{q} = q' + v'$}; 
    \end{tikzpicture}
    \caption{\label{fig:illustration_non_reversibility} Illustration of the need for a reverse projection check, for $V=0$, $M=\mathrm{Id}$ and a projection defined using the closest point to the submanifold. Starting from a position~$q$ and a momentum~$p$, an unconstrained update of the position leads to~$\tilde{q} = q + v$ with $v = \Delta t \, M^{-1}p$, which is then projected back onto the submanifold in the direction $M^{-1} \nabla \xi(q)$ as $q'=\tilde{q} + M^{-1} \nabla \xi(q) \Lambda(q,\tilde{q})$. The associated momentum is denoted by~$p'$. Starting from~$q'$ with momentum~$-p'$, the reverse move proceeds by an unconstrained update of the position, leading to~$\widehat{q} = q' + v'$ with $v' = -\dt M^{-1}p'$. On this example, the closest point from this one on the submanifold in the direction $M^{-1} \nabla \xi(q')$ is $q''$, which is different from~$q$. }
  \end{center}
\end{figure}
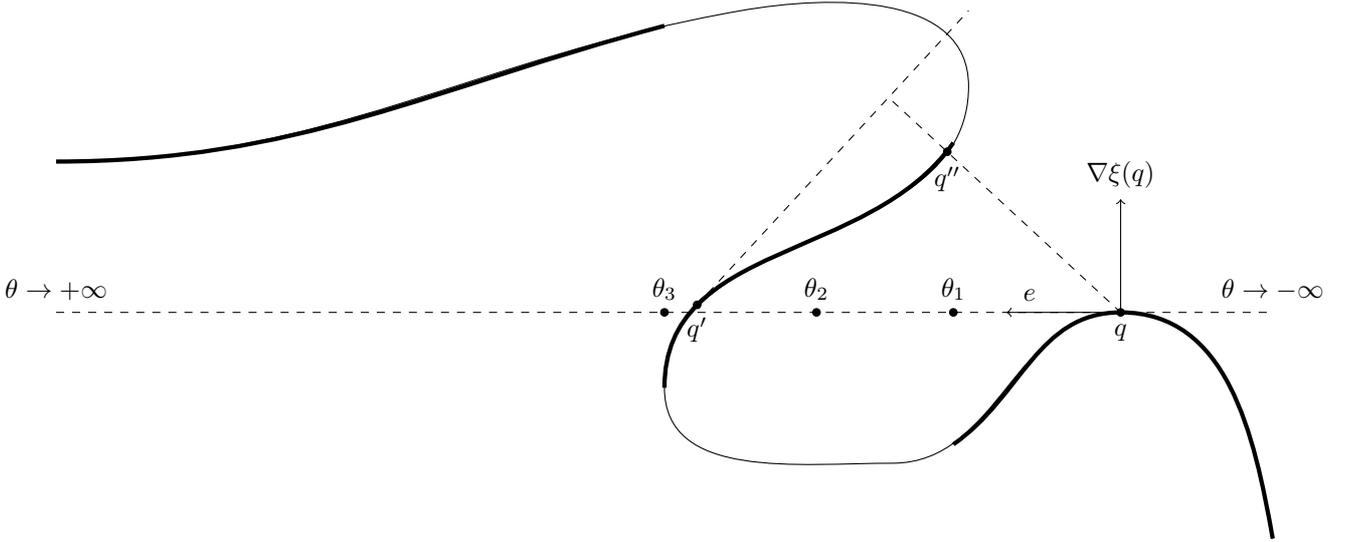
\begin{figure}
  \begin{center}
    \begin{tikzpicture}
      \draw [dashed] (-1,0) -- (15,0);
      \draw [->] (13,0) -- (13,1.5);
      \draw (13,1.5) node [black,above=1] {$\nabla \xi(q)$};
      \draw (-1,2)
      to [in=190,out=0] (8,4)
      to [in=90,out=10] (11,3)
      to [in=90,out=270] (7,-1)
      to [in=180,out=270] (10,-2)
      to [in=180,out=0] (13,0)
      to [in=100,out=0] (15,-3);
      \draw (15,0) node [black,above=1] {$\theta \to -\infty$};
      \draw (-1,0) node [black,above=1] {$\theta \to +\infty$};
      \draw [fill] (13,0) circle [radius=0.05] node [black,below=1] {$q$};
      \draw [->] (13,0) -- (11.5,0);
      \draw (11.8,0) node [black,above=1] {$e$};
      \draw [fill] (10.8,0) circle [radius=0.05] node [black,above=1] {$\theta_1$};
      \draw [fill] (9,0) circle [radius=0.05] node [black,above=1] {$\theta_2$};
      \draw [fill] (7,0) circle [radius=0.05] node [black,above=1] {$\theta_3$};
      \draw [ultra thick] (-1,2)
      to [in=195,out=0] (7,3.8);
      \draw [ultra thick] (7,-1)
      to [in=235,out=90] (10.8,2.25);
      \draw [ultra thick] (10.8,-1.75)
      to [in=180,out=35] (13,0)
      to [in=100,out=0] (15,-3);
      \draw [fill] (7.43,0.1) circle [radius=0.05] node [black,below=1] {$q'$};
      \draw [fill] (10.72,2.13) circle [radius=0.05] node [black,below=2] {$q''$};
      \draw [dashed] (7.43,0.1) -- (11,4);
      \draw [dashed] (10,2.8) -- (13,0);
    \end{tikzpicture}
    \caption{\label{fig:illustration_non_C1} For a given initial position~$q$, and various increments~$v_\theta = \Delta t M^{-1} p_\theta$ with $p_\theta = \theta \, e$ and $e$ a unit vector in $T^*_{q}\mathcal{M}$ (with $m=1$, $M = \mathrm{Id}$, $V=0$), the projection of $\tilde{q} = q+v$ onto the submanifold (thick black lines), chosen by minimizing the distance to the submanifold, can change abruptly. This is the case here around~$v_{\theta_1}$ and~$v_{\theta_3}$. For $\theta < \theta_1$, the resulting projection $q+v_\theta+M^{-1}\nabla \xi(q)\Lambda(q,q+v_\theta)$ is in the bottom part of the submanifold; for $\theta_1 < \theta < \theta_3$, it is in the middle part; while for $\theta > \theta_3$ it is on the upper part. Inverting momenta and performing one step of the RATTLE scheme from the projected configurations obtained with $\theta < \theta_2$ or $\theta > \theta_3$ allows to come back to the initial configuration~$(q,p)$. However, for increments $v_\theta$ with $\theta \in (\theta_2,\theta_3)$, the so-obtained projected position is different from~$q$, and lies on the upper side of the submanifold when choosing the projection to be the closest position on the submanifold (see how $q''$ is obtained starting from~$q'$). The increment~$v_{\theta_2}$ corresponds to the projected position where the normal is parallel to the tangent to the submanifold passing through~$q$ and a point on the upper part of the submanifold. In this example, the reverse projection check is not successful for increments corresponding to $\theta \in (\theta_2,\theta_3)$, while it is successful for small increments (corresponding to $\theta < \theta_2$) or for sufficiently large ones (corresponding to $\theta > \theta_3$).}
  \end{center}
\end{figure}

The following result shows that $B$ is an open set, obtained as the union of open connected components of $A \cap \Psi_{\Delta t}^{-1}(A)$.  We provide in Figure~\ref{fig:illustration_non_C1}  an illustration of momenta~$p$ belonging to~$B$ or not, for a given~$q \in \mathcal{M}$. The fact that $B$ is open is an important property for proving the correctness of the sampling method (see the proof of Proposition~\ref{prop:PsiA}). Before stating the result, we recall that a subset $S_0 \subset S$ is a path connected component of $S$ if any element of $S$ which can be connected to an element of $S_0$ by a continuous path in $S$, also belongs to $S_0$ (and in fact all elements along the path are in~$S_0$). Since $T^* \mathcal{M}$ is a topological manifold, it is locally path connected, and hence the path connected components of any open subset of $T^* \mathcal{M}$ are exactly its connected components, and these connected components are open subsets of $T^* \mathcal{M}$ (see for instance~\cite[Theorem~2.9.22]{Schwartz1}).

\begin{lem}
  \label{lem:B} 
  Let $C$ be a path connected component of $A \cap \Psi^{-1}_{\Delta t}(A)$. If there is $(q,p) \in C$ such that $(\Psi_{\Delta t}\circ \Psi_{\Delta t})(q,p) = (q,p)$, then 
  \begin{equation}\label{eq:PC}
  \forall (q,p) \in C, \qquad (\Psi_{\Delta t}\circ \Psi_{\Delta t})(q,p) = (q,p).
  \end{equation}
  As a corollary, the set $B$ defined by~\eqref{eq:B}, namely
  \[
  B = \Big\{ (q,p) \in A \cap \Psi_{\Delta t}^{-1}(A), \ (\Psi_{\Delta t}\circ \Psi_{\Delta t})(q,p) = (q,p) \Big\}
  \]
  is the union of path connected components of the open set $A \cap \Psi_{\Delta t}^{-1}(A)$. In particular, it is an open set of $T^*\mathcal{M}$.
\end{lem}

A direct corollary of Proposition~\ref{prop:psi_Dt} and Lemma~\ref{lem:B} is the following result, where we denote by $B^c = T^*\mathcal{M} \backslash B$ the complement of~$B$ in $T^*\mathcal{M}$.

\begin{prop}
  \label{prop:PsiA}
  The map $\Psi_{\Delta t}^{\rm rev}: T^* \mathcal M \to T^* \mathcal M$ defined by~\eqref{eq:PsiB} is globally well defined, and satisfies
$$\Psi_{\Delta t}^{\rm rev} \circ \Psi_{\Delta t}^{\rm rev} = {\rm Id}.$$
Moreover, both $\Psi_{\Delta t}^{\rm rev}:B \to B$ and  $\Psi_{\Delta t}^{\rm rev}:B^c \to B^c$ are $C^1$-diffeomorphisms which preserve the measure  $\sigma_{T^* {\mathcal M}} (dq \, dp)$.
As a consequence, $\Psi_{\Delta t}^{\rm rev}: T^* \mathcal M \to T^* \mathcal M$ globally
 preserves the measure $\sigma_{T^* {\mathcal M}} (dq \, dp)$.
\end{prop}
\begin{proof}
  Since $B$ is open and $\Psi_{\Delta t}$ is a measure-preserving local $C^1$ diffeomorphism and an involution on~$B$, this mapping is a global measure-preserving  $C^1$ diffeomorphism from $B$ to $B$. Moreover, since $\Psi_{\Delta t}$ is the identity map on $B^c$, it is also an involution, a global measure-preserving and a $C^1$ diffeomorphism from $B^c$ to $B^c$. Therefore, $\Psi_{\Delta t}^{\rm rev} \circ \Psi_{\Delta t}^{\rm rev} = \mathrm{Id}$ by distinguishing the cases when $(q,p) \in B^c$, which is trivial; or $(q,p) \in B$ (in which case we rely on Lemma~\ref{lem:B}). 

Moreover, for any measurable set $S \subset T^*\mathcal{M}$, and using the notation $|S|= \int_{T^* {\mathcal M}} 1_{S}(q,p) \sigma_{T^* {\mathcal M}} (dq \, dp)$,
  \[
  \left|(\Psi_{\Delta t}^{\rm rev})^{-1}(S)\right| = \left|(\Psi_{\Delta t}^{\rm rev})^{-1}(S \cap B)\right| + \left|(\Psi_{\Delta t}^{\rm rev})^{-1}(S \cap B^c)\right| = \left|(\Psi_{\Delta t})^{-1}(S \cap B)\right| + \left|S \cap B^c\right| = \left|S \cap B\right| + \left|S \cap B^c\right| = |S|,
  \]
  where we used the definition of~$\Psi_{\Delta t}^{\rm rev}$ to obtain the second equality, and the measure preserving properties of~$\Psi_{\Delta t}$ on~$B$ for the third one. This proves that $\Psi_{\Delta t}^{\rm rev}$ preserves the measure $\sigma_{T^* {\mathcal M}} (dq \, dp)$.
\end{proof}

Notice the importance of the reverse projection check to prove that $\Psi_{\Delta t}^{\rm rev}$ is indeed an involution. Let us also emphasize that it is crucial in the above proof of the invariance of the Lebesgue measure (see Proposition~\ref{prop:PsiA}) that $B$ is an open set, in order to perform the change of variable argument.

Let us now conclude this section by providing the proof of Lemma~\ref{lem:B}. The key idea is that points in $B$ for which it is not possible to build a neighborhood in $B$ are necessarily configurations where the non-tangential condition~\eqref{eq:non_tang_q^1} is not satisfied. Indeed, this is would mean that it is possible to construct two paths in configuration space originating from such points which smoothly branch out to different projections (namely different solutions of~\eqref{eq:Pi}; see Figure~\ref{fig:illustration_proof}). This is why we excluded in the Definition~\ref{def:proj_sel} of the Lagrange multiplier points where the non-tangential condition is not satisfied. The non-tangential condition~\eqref{eq:non_tang_q^1} is therefore a sufficient condition to prove that the set $B$ is indeed an open set, as stated in Lemma~\ref{lem:B}.

\begin{proof}[Proof of Lemma~\ref{lem:B}]
  Note that $A \cap \Psi_{\Delta t}^{-1}(A)$ is open by continuity of $\Psi_{\Delta t}$ so that the results stated in Lemma~\ref{lem:B} are a direct consequence of~\eqref{eq:PC}. Let us prove~\eqref{eq:PC}, arguing by contradiction. Consider a non-empty connected component~$C$ of $A \cap \Psi_{\Delta t}^{-1}(A)$, and an element $(q_0,p_0) \in C$ such that $(\Psi_{\Delta t}\circ \Psi_{\Delta t})(q_0,p_0) = (q_0,p_0)$. Note that $\Psi_{\Delta t}\circ \Psi_{\Delta t}$ is indeed well defined on $A \cap \Psi_{\Delta t}^{-1}(A)$. Assume that there exists $(q_1,p_1) \in C$ such that $(\Psi_{\Delta t}\circ \Psi_{\Delta t})(q_1,p_1) \neq (q_1,p_1)$. Introduce a continuous path $[0,1] \ni \theta \mapsto (q_\theta,p_\theta)$ with values in $A \cap \Psi_{\Delta t}^{-1}(A)$ linking $(q_0,p_0)$ to $(q_1,p_1)$. We will prove that there exists $\theta_\star \in (0,1)$ such that the non-tangential property~\eqref{eq:non_tang_q^1} is violated at~$(q,p)=(q_{\theta_\star}, p_{\theta_\star})$, which is in contradiction with $(q_{\theta_\star},p_{\theta_\star}) \in A$. See Figure~\ref{fig:illustration_proof} for a schematic representation of the various objects introduced in this proof. 

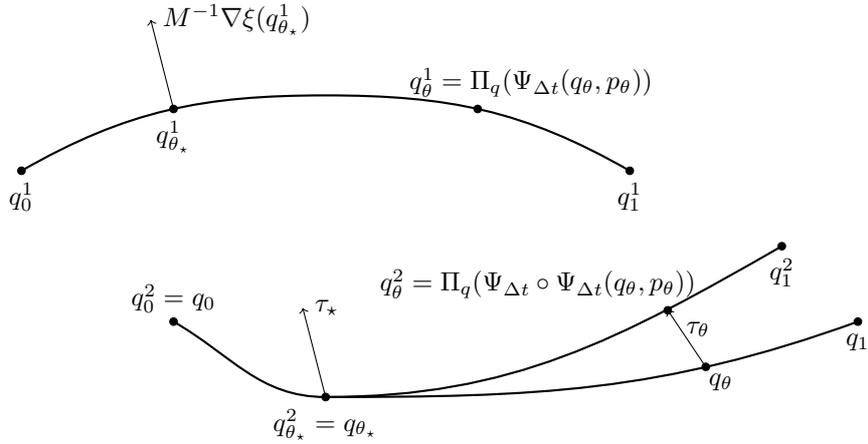
\begin{figure}
  \begin{center}
    \begin{tikzpicture}
      \draw [thick] (0,3) 
      to [out=30,in=180] (4,4)
      to [out=0,in=150] (8,3);
      \draw [fill] (0,3) circle [radius=0.05] node [black,below=1] {$q_0^1$};
      \draw [fill] (8,3) circle [radius=0.05] node [black,below=1] {$q_1^1$};
      \draw [fill] (6,3.82) circle [radius=0.05] node [black,above=1] {\qquad\qquad $q_\theta^1 = \Pi_q (\Psi_{\Delta t}(q_\theta,p_\theta))$};
      \draw [fill] (2,3.82) circle [radius=0.05] node [black,below=1] {$q_{\theta_\star}^1$};
      \draw [->] (2,3.82) -- (1.7,5);
      \draw (1.7,5) node [black,right=1] {$M^{-1}\nabla \xi(q_{\theta_\star}^1)$};
      \draw [thick] (2,1) 
      to [out=-30,in=180] (4,0)
      to [out=0,in=210] (10,2);
      \draw [thick] (4,0)
      to [out=0,in=200] (11,1);
      \draw [fill] (10,2) circle [radius=0.05] node [black,below=1] {$q_1^2$};
      \draw [fill] (11,1) circle [radius=0.05] node [black,below=1] {$q_1$};
      \draw [fill] (8.5,1.15) circle [radius=0.05] node {};
      \draw (9.1,1.5) node [black,left=3] {$q_\theta^2 = \Pi_q (\Psi_{\Delta t} \circ \Psi_{\Delta t}(q_\theta,p_\theta))$};
      \draw [fill] (9,0.4) circle [radius=0.05] node {};
      \draw (9.2,0.2) node [black] {$q_\theta$};
      \draw [->] (9,0.4) -- (8.5,1.15);
      \draw (8.9,0.9) node [black] {$\tau_\theta$};
      \draw [fill] (2,1) circle [radius=0.05] node [black,above=0] {$q_0^2 = q_0$};
      \draw [fill] (4,0) circle [radius=0.05] node [black,below=1] {$q^2_{\theta_\star} = q_{\theta_\star}$};
      \draw [->] (4,0) -- (3.7,1.18); 
      \draw (3.7,1.18) node [black,right=1] {$\tau_\star$};
    \end{tikzpicture}
    \caption{\label{fig:illustration_proof} Schematic representation of the objects introduced in the proof of Lemma~\ref{lem:B}. The projection $\Pi_q$ is defined as $\Pi_q x = q$ for $x=(q,p) \in \R^d \times \R^d$.}
  \end{center}
\end{figure}

Define
\[
\theta_\star = \sup \Big\{ \theta \geq 0 \, \Big| \, \forall s \in [0,\theta], \ \left(\Psi_{\Delta t}\circ\Psi_{\Delta t}\right)(q_s,p_s) = (q_s,p_s) \Big\}.  
\]
In fact, the supremum is a maximum (by continuity of $\Psi_{\Delta t}\circ\Psi_{\Delta t}$ and of $\theta \mapsto (q_\theta,p_\theta)$), and $\theta_\star < 1$. By definition of the supremum, and introducing $(q_{\theta_n}^2,p_{\theta_n}^2) = (\Psi_{\Delta t}\circ \Psi_{\Delta t})(q_{\theta_n},p_{\theta_n})$, there exists a sequence $(\theta_n)_{n \geq 0}$ with values in $(\theta_\star,1)$ such that $\theta_n \to \theta_\star$ as $n \to +\infty$, and  
\begin{equation}
  \label{eq:not_involution_condition}
  \forall n \geq 0, \qquad  (q_{\theta_n}^2,p_{\theta_n}^2) \neq (q_{\theta_n},p_{\theta_n}). 
\end{equation}
Denoting by $(q^1_{\theta_n},-p_{\theta_n}^1) = \Psi_{\Delta t}(q_{\theta_n},p_{\theta_n})$ (so that $(q^2_{\theta_n},p_{\theta_n}^2) = \Psi_{\Delta t}(q^1_{\theta_n},-p^1_{\theta_n})$), there are Lagrange multipliers $\lambda_{\theta_n}^{1/2},\lambda_{\theta_n}^{1},\lambda_{\theta_n}^{3/2},\lambda_{\theta_n}^2$ such that
\begin{equation}
  \label{eq:expressions_q_p}
  \left\{ \begin{aligned}
    q_{\theta_n}^1 & = q_{\theta_n} + \Delta t M^{-1}\left( p_{\theta_n} - \frac{\Delta t}{2} \nabla V(q_{\theta_n})\right) + \Delta t M^{-1} \nabla \xi(q_{\theta_n}) \lambda_{\theta_n}^{1/2}, \\
    p_{\theta_n}^1 & = p_{\theta_n} - \frac{\Delta t}{2}\left( \nabla V(q_{\theta_n}) + \nabla V(q_{\theta_n}^1) \right) + \nabla \xi(q_{\theta_n}) \, \lambda_{\theta_n}^{1/2} + \nabla \xi(q_{\theta_n}^1) \, \lambda_{\theta_n}^{1}, \\
    q_{\theta_n}^2 & = q_{\theta_n}^1 - \Delta t M^{-1}\left( p_{\theta_n}^1 + \frac{\Delta t}{2} \nabla V(q_{\theta_n}^1)\right) + \Delta t M^{-1} \nabla \xi(q_{\theta_n}^1) \lambda_{\theta_n}^{3/2}, \\
    p_{\theta_n}^2 & = p_{\theta_n}^1 + \frac{\Delta t}{2}\left( \nabla V(q_{\theta_n}^1) + \nabla V(q_{\theta_n}^2) \right) - \nabla \xi(q_{\theta_n}^1) \, \lambda_{\theta_n}^{3/2} - \nabla \xi(q_{\theta_n}^2) \, \lambda_{\theta_n}^{2}.
  \end{aligned} \right.
\end{equation}
Using the first two equations above, it is possible to express $q_{\theta_n}$ in terms of $(q^1_{\theta_n},p_{\theta_n}^1)$ as follows:
\begin{equation}
  \label{eq:q_reformulated}
  q_{\theta_n} = q_{\theta_n}^1 - \Delta t M^{-1}\left( p_{\theta_n}^1 + \frac{\Delta t}{2} \nabla V(q_{\theta_n}^1)\right) + \Delta t M^{-1} \nabla \xi(q_{\theta_n}^1) \lambda_{\theta_n}^{1}.
\end{equation}

Note now that $(q^2_{\theta_n},p^2_{\theta_n}) \neq (q_{\theta_n},p_{\theta_n})$ implies that  $q^2_{\theta_n} \neq q_{\theta_n}$. Indeed, if this was not the case (namely if $q^2_{\theta_n} = q_{\theta_n}$), then the third equation in~\eqref{eq:expressions_q_p} and~\eqref{eq:q_reformulated} would imply that $M^{-1}\nabla \xi(q_{\theta_n}^1) \left(\lambda_{\theta_n}^{1} - \lambda_{\theta_n}^{3/2}\right) = 0$. Therefore, upon multiplying on the left by $\left[\nabla \xi(q_{\theta_n}^1)\right]^T$, 
\[
G_M(q_{\theta_n}^1)\left(\lambda_{\theta_n}^{1} -\lambda_{\theta_n}^{3/2}\right) = 0.
\]
The invertibility of $G_M$ allows to conclude that $\lambda_{\theta_n}^{1} = \lambda_{\theta_n}^{3/2}$. Then, since we assumed that $q^2_{\theta_n} = q_{\theta_n}$,
\[
p_{\theta_n}^2 - p_{\theta_n} = \nabla \xi(q_{\theta_n}) \left( \lambda_{\theta_n}^{1/2}-\lambda_{\theta_n}^{2}\right).
\]
Since $p_{\theta_n}^2-p_{\theta_n} \in T^*_{q_{\theta_n}} \mathcal{M}$, it holds $\left[\nabla \xi(q_{\theta_n})\right]^T M^{-1} (p_{\theta_n}^2-p_{\theta_n}) = 0 = G_M(q_{\theta_n}) ( \lambda_{\theta_n}^{1/2}-\lambda_{\theta_n}^{2})$, which shows that $\lambda_{\theta_n}^{1/2} = \lambda_{\theta_n}^{2}$. Therefore, $p_{\theta_n}^2 = p_{\theta_n}$, in contradiction with $(q^2_{\theta_n},p^2_{\theta_n}) \neq (q_{\theta_n},p_{\theta_n})$. This concludes the proof that $q^2_{\theta_n} \neq q_{\theta_n}$.

In view of~\eqref{eq:not_involution_condition} and the above argument, we can define the following family of unit vectors:
\[
\forall n \geq 0, \qquad \tau_{\theta_n} = \frac{q_{\theta_n}^2 - q_{\theta_n}}{\left\| q_{\theta_n}^2 - q_{\theta_n} \right\|},
\]
where $\|\cdot\|$ is the Euclidean norm on~$\mathbb{R}^d$. Note in particular that the denominator is positive. By compactness, and upon extraction (although, with some abuse of notation, we keep the same notation for the extracted subsequence), there exist a subsequence $(\theta_n)_{n \geq n_\star}$ and a unit vector $\tau_\star$ such that $\tau_{\theta_n} \to \tau_\star$ as $n \to +\infty$. By construction, $M \tau_\star$ belongs to $T_{q_{\theta_\star}}^* \mathcal{M}$, the cotangent space of $\mathcal M$ at $q_{\theta_\star}$ (see Figure~\ref{fig:illustration_proof}). The latter assertion follows from the fact that $q_{\theta_\star}^2 = q_{\theta_\star}$, so that $q_{\theta_n}^2 -q_{\theta_n} = \left(q_{\theta_n}^2 - q_{\theta_\star}^2\right) - \left(q_{\theta_n}-q_{\theta_\star}\right)$. Moreover, $(q_{\theta_n}-q_{\theta_\star})/\|q_{\theta_n}-q_{\theta_\star}\|$ converge to some element in $T^*_{q_{\theta_\star}} \mathcal{M}$ by definition of the tangent space, and so does $(q_{\theta_n}^2-q_{\theta_\star}^2)/\|q_{\theta_n}^2-q_{\theta_\star}^2\|$. This shows finally that if $\tau_{\theta_n}$ has a limit, this limit necessarily belongs to $T_{q_{\theta_\star}}^* \mathcal{M}$.

The expressions of $q^2_{\theta_n}$ and $q_{\theta_n}$ in~\eqref{eq:expressions_q_p}-\eqref{eq:q_reformulated} show that $\tau_{\theta_n} = M^{-1} \nabla \xi(q_{\theta_n}^1) \Lambda_{n}$ for some $\Lambda_{n} \in \mathbb{R}^m$. Note that the norm of $\Lambda_{n}$ is uniformly bounded away from~0 since $\|\tau_{\theta_n}\|=1$. Moreover, $\Lambda_{n} = G_M(q_{\theta_n}^1)^{-1} \left[\nabla \xi(q_{\theta_n}^1)\right]^T \tau_{\theta_n}$. Since $\tau_{\theta_n} \to \tau_\star$, the functions $G_M$ and $\nabla \xi$ are continuous, and $q_{\theta_n}^1 \to q^1_{\theta_\star}$ where $(q^1_{\theta_\star},p^1_{\theta_\star}) = \Psi_{\Delta t}(q_{\theta_\star},p_{\theta_\star})$ (because $\Psi_{\Delta t}$ is continuous and ($q_{\theta_n},p_{\theta_n}) \to (q_{\theta_\star},p_{\theta_\star})$), it follows that $\Lambda_n$ converges to some vector $\Lambda_\star$ in the limit $n \to \infty$, where 
\[
\Lambda_\star = G_M(q_{\theta_\star}^1)^{-1} \nabla \left[\xi(q_{\theta_\star}^1)\right]^T \tau_\star \in \mathbb{R}^d \backslash \{0\}, 
\qquad
\tau_\star = M^{-1} \nabla \xi(q_{\theta_\star}^1) \Lambda_\star. 
\]
Since $M \tau_\star$ belongs to $T_{q_{\theta_\star}}^* \mathcal{M}$, it holds $\left[\nabla \xi(q_{\theta_\star})\right]^T \tau_\star = 0 = \left[\nabla \xi(q_{\theta_\star})\right]^T M^{-1} \nabla \xi(q^1_{\theta_\star}) \Lambda_\star$. This in contradiction with the non-tangential projection property~\eqref{eq:non_tang_q^1} at~$(q,p)=(q_{\theta_\star}, p_{\theta_\star})$ (which is a direct consequence of~\eqref{eq:non_tang}) because $\Lambda_\star \neq 0$. The contradiction shows that there is no element $(q_1,p_1) \in C$ such that $(\Psi_{\Delta t}\circ \Psi_{\Delta t})(q_1,p_1) \neq (q_1,p_1)$, which concludes the proof of~\eqref{eq:PC}.
\end{proof}

\subsection{Examples of admissible Lagrange multiplier functions}
\label{sec:example}

We discuss in this section possible definitions of admissible Lagrange multiplier functions, with their associated ensembles~$A$ and~$B$. For simplicity, we assume in all this section that 
\begin{equation}\label{eq:Mcompact}
\exists \alpha >0, \qquad \{q \in \R^d, \|\xi(q)\| \le \alpha \} \text{ is compact}.
\end{equation}
In particular, $\mathcal{M}$ is compact.  In the case when $\mathcal{M}$ is not compact, one can obtain local versions of the results below by localization arguments (namely considering trajectories of the Markov chain up to the first time they leave a compact subset of $\mathcal{M}$). 

\subsubsection{Example 1: A local and theoretical projection based on the implicit function theorem}
\label{sec:ex1_}

A first way to construct a Lagrange multiplier function is to rely on the implicit function theorem to define it in a neighborhood of the submanifold $T^*\mathcal{M}$. The discussion here is reminiscent of what can be read for example in~\cite[Section~VII.1.3]{hairer-lubich-wanner-06} where the implicit function theorem is also used to prove the well-posedness and symmetry of RATTLE. However, our presentation differs from this work since we want to define the Lagrange multiplier function globally, and not only on a single timestep. This is necessary to check the properties stated in Definition~\ref{def:proj_sel}. Moreover, we also would like to possibly use large timesteps, and thus define $\Lambda(q,\tilde{q})$ for $\tilde{q}$ which is not in an arbitrarily small neighborhood of $q$.

\paragraph{Construction of the Lagrange multiplier function.}
Here and in the following, we use the following notation: for any $(q,\tilde{q}) \in \R^d \times \R^d$, we define the matrix
\begin{equation}
  \label{eq:GMqq}
  \mathcal{G}_M(q, \tilde{q})=[\nabla \xi(q)]^T M^{-1} \nabla \xi(\tilde{q}) \in \R^{m \times m}.
\end{equation}
Notice that this is simply the bilinear form associated with the quadratic form~$G_M$ defined in~\eqref{eq:GM}. The main result of this section is the following proposition. 

\begin{prop}
  \label{prop:proj_well_defined_}
  Assume that~\eqref{eq:Mcompact} holds. There exists an open subset $\mathcal D_{\rm imp}$ of $\mathcal M \times \R^d$ and an admissible (in the sense of Definition~\ref{def:proj_sel}) Lagrange multiplier function $\Lambda: \mathcal D_{\rm imp} \to \R^m$ such that the following properties hold:
\begin{itemize}
\item The set $\mathcal D_{\rm imp}$ contains $\mathcal E$ and $\Lambda$ is zero on $\mathcal E$ where
\begin{equation}\label{eq:E}
\mathcal E=\Big\{ (q,\tilde{q}) \in \mathcal M \times \mathcal M, \ \mathcal{G}_M(q,\tilde{q}) \text{ is invertible} \Big\};
\end{equation}
\item For any $(q,\tilde{q}) \in \mathcal D_{\rm imp}$, the matrix $\mathcal{G}_M(q,\tilde{q})$ is invertible;
\item For any $(q_0,\tilde{q}_0) \in \mathcal D_{\rm imp}$, there exists a neighborhood $\mathcal V_0$ of $(q_0,\tilde{q}_0)$ in $\mathcal D_{\rm imp}$ and a positive real number~$\alpha_0$ such that
  \begin{equation}
    \label{eq:lambda_minoree_}
\forall (q,\tilde{q}) \in \mathcal V_0, \,
\|\Lambda(q,\tilde q)\| < \alpha_0 \text{ and }    \forall \lambda \in \R^m \setminus \{\Lambda(q,\tilde{q})\}, \, \xi(\tilde{q}+M^{-1} \nabla \xi (q) \lambda)=0 \implies \|\lambda \| \ge \alpha_0.
  \end{equation}
\end{itemize}
\end{prop}

Let us emphasize that we do not require $q$ and $\tilde{q}$ to be close one from another in this result. Notice also that a consequence of~\eqref{eq:lambda_minoree_} is that the Lagrange multiplier $\Lambda(q,\tilde{q})$ is simply the solution with the smallest norm to the following problem: find $\lambda \in \R^m$ such that $\xi(\tilde{q}+M^{-1} \nabla \xi (q) \lambda)=0$. 

Before proving Proposition~\ref{prop:proj_well_defined_}, let us give two examples of explicit subsets of $\mathcal D_{\rm imp}$. Here and in the following, we use the following notation: for any integer $k$, for any $q \in \R^k$ and $r>0$, we denote by $\mathcal B_k(q,r) \subset \R^k$ the ball centered at $q$ with radius $r$.
\begin{prop}\label{lem:DDtilde}
The two following properties hold:
\begin{enumerate}[(i)]
\item There exists a non increasing function $\delta: \R_+ \to \R_+^*$ such that, for any $(q, \tilde{q}) \in \mathcal M \times \R^d$, if $\mathcal{G}_M(q,\tilde{q})$ is invertible, and $\|\xi(\tilde{q})\| < \delta(\|\mathcal{G}_M(q,\tilde{q})^{-1}\|)$, then $(q,\tilde{q}) \in \mathcal D_{\rm imp}$.
\item There exists $\beta_0>0$ such that, for any $q \in \mathcal M$ and any $\tilde{q} \in \mathcal B_{d}(q,\beta_0)$, it holds $(q, \tilde{q}) \in \mathcal D_{\rm imp}$. Moreover, $\beta_0>0$ can be chosen such that there exists $\alpha_1>0$ such that,  for any $q \in \mathcal M$ and any $\tilde{q} \in \mathcal B_{d}(q,\beta_0)$,
 \begin{equation}
    \label{eq:lambda_minoree}
\|\Lambda(q,\tilde q)\| < \alpha_1 \text{ and }   \forall \lambda \in \R^m \setminus \{\Lambda(q,\tilde{q})\}, \, \xi(\tilde{q}+M^{-1} \nabla \xi (q) \lambda)=0 \implies 
    \|\lambda \| \ge \alpha_1.
  \end{equation}
\end{enumerate}
\end{prop}

Let us emphasize that in the first item, we do not assume that $q$ and $\tilde{q}$ are close (think for example of points on opposite sides of a sphere), but $\tilde{q}$ should be sufficiently close from the submanifold and the inverse of the matrix $\mathcal{G}_M(q,\tilde{q})$ should be not too large in order for the projection to be well defined. This first item will be useful to formalize the practical approach to define the projection, using Newton's algorithm (see Section~\ref{sec:ex2}). The second item means that the projection can be defined in open balls of uniform sizes around any point on the submanifold. This will be useful to analyze the properties of the numerical scheme based on the Lagrange multiplier function $\Lambda$ introduced here, see Section~\ref{sec:rattle_imp}.

\paragraph{Proofs of Propositions~\ref{prop:proj_well_defined_} and~\ref{lem:DDtilde}.}

The end of this section is devoted to the proofs of  Propositions~\ref{prop:proj_well_defined_} and~\ref{lem:DDtilde}.
The proof of Proposition~\ref{prop:proj_well_defined_} requires a preliminary result which shows that it is possible to construct the Lagrange multiplier function on a larger open set~$\widetilde{\mathcal{D}}_\zeta$ where both $q$ and $\tilde{q}$ are in a neighborhood of~$\mathcal{M}$ (but not necessarily close one from another) and such that $\mathcal{G}_M(q,\tilde{q})$ is invertible. 

\begin{lem}
  \label{prop:proj_well_defined__}
  Assume that~\eqref{eq:Mcompact} holds. For any positive $\zeta$, define
  \begin{equation}\label{eq:E_zeta}
    \mathcal E_\zeta=\left\{(q,\tilde{q}) \in \mathcal M \times \mathcal M, \ \mathcal{G}_M(q,\tilde{q}) \text{ is invertible and } \left\| \mathcal{G}_M(q,\tilde{q})^{-1}\right\| \le \zeta \right\},
\end{equation}
  which is non-empty for $\zeta \geq \min_{\mathcal{M}} \|G_M^{-1}\|$.
  There exists an open set $\widetilde{\mathcal D}_{\zeta}$ of $\R^d \times \R^d$ such that $\mathcal E_\zeta \subset \widetilde{\mathcal D}_{\zeta}$, and a smooth Lagrange multiplier function $\Lambda: \widetilde{\mathcal D}_{\zeta} \to \R^m$ such that~\eqref{eq:Pi} and~\eqref{eq:non_tang} are satisfied with $\mathcal{D}$ replaced by~$\widetilde{\mathcal D}_\zeta$. Moreover, $\widetilde{\mathcal D}_{\zeta}$ can be chosen such that there exists $\widetilde \alpha(\zeta) > 0$ for which, for any $(q,\tilde{q}) \in \tilde{\mathcal D}_{\zeta}$, it holds 
  \begin{equation}
    \label{eq:lambda_minoree__}
\|\Lambda(q,\tilde q)\| < \tilde \alpha(\zeta) \text{ and }    \forall \lambda \in \R^m \setminus \{\Lambda(q,\tilde{q})\}, \, \xi(\tilde{q}+M^{-1} \nabla \xi (q) \lambda)=0 \implies 
    \|\lambda \| \ge \tilde \alpha(\zeta).
  \end{equation}
\end{lem}

\begin{proof}
Notice that for $\zeta>0$, the set $\mathcal E_\zeta$ defined by~\eqref{eq:E_zeta} is a compact set. Moreover, it is non empty for sufficiently large~$\zeta$, since, for any $q \in \mathcal M$, $\mathcal{G}_M(q,q)=G_M(q)$ is invertible (see~\eqref{eq:GM}), so that $(q,q) \in \mathcal E_\zeta$ for any $\zeta \ge \|G_M(q)^{-1}\| $. Consider now, for a fixed $\zeta>0$, a couple $(q_0,\tilde{q}_0) \in \mathcal E_\zeta$ and let us introduce the function 
\[
F:\left\{\begin{aligned}
\R^d \times \R^d \times \R^m &\to \R^m\\
(q,\tilde{q},\lambda) &\mapsto \xi(\tilde{q}+M^{-1} \nabla \xi(q) \lambda).
\end{aligned}
\right.
\]
This is a smooth function such that $F(q_0,\tilde{q}_0,0)=0$, and
\[
\frac{\partial F}{\partial \lambda}(q_0,\tilde{q}_0,0)=[\nabla \xi (\tilde{q}_0)]^T M^{-1} \nabla \xi(q_0)
\]
is invertible by assumption. Therefore, by the implicit function theorem, there exist $\eta_0(q_0,\tilde{q}_0)>0$, $\alpha_0(q_0,\tilde{q}_0)>0$ and a smooth function $\Lambda:\mathcal B_d(q_0,\eta_0(q_0,\tilde{q}_0)) \times \mathcal B_d(\tilde{q_0},\eta_0(q_0,\tilde{q}_0)) \to \mathcal B_m(0,\alpha_0(q_0,\tilde{q}_0))$ such that, for all $(q,\tilde{q}) \in \mathcal B_d(q_0,\eta_0(q_0,\tilde{q}_0)) \times \mathcal B_d(\tilde{q_0},\eta_0(q_0,\tilde{q}_0))$ and all $\lambda \in  \mathcal B_m(0,\alpha_0(q_0,\tilde{q}_0))$,
\begin{equation}\label{eq:lambda_iff}
\xi(\tilde{q}+M^{-1} \nabla \xi(q) \lambda)=0 \iff \lambda=\Lambda(q,\tilde{q}).
\end{equation}
By a continuity argument, upon reducing $\eta_0(q_0,\tilde{q}_0)$, one can assume that the non-tangential projection property~\eqref{eq:non_tang} holds:
for all $(q,\tilde{q}) \in \mathcal B_d(q_0,\eta_0(q_0,\tilde{q}_0)) \times \mathcal B_d(\tilde{q_0},\eta_0(q_0,\tilde{q}_0))$,
\[
\left[\nabla \xi\left( \tilde{q} + M^{-1} \nabla \xi(q) \Lambda(q,\tilde{q})\right)\right]^T M^{-1} \nabla \xi(q) \text{ is invertible},
\]
since this holds for $(q,\tilde{q})=(q_0,\tilde{q}_0)$.

Remark now that $\dps \bigcup_{(q_0,\tilde{q}_0) \in \mathcal E_\zeta} \mathcal B_d(q_0,\eta_0(q_0,\tilde{q}_0)) \times \mathcal B_d(\tilde{q}_0,\eta_0(q_0,\tilde{q}_0))$ is an open cover of the compact set $\mathcal E_\zeta$, from which it is possible to extract a finite cover:
\[
\mathcal E_\zeta \subset \bigcup_{i=1}^n \mathcal B_d(q_i,\eta_0(q_i,\tilde{q}_i)) \times \mathcal B_d(\tilde{q_i},\eta_0(q_i,\tilde{q}_i)),
\]
where $(q_i,\tilde{q}_i)_{i=1, \ldots,n}$ are elements of $\mathcal E_\zeta$. This motivates introducing the open set
\[
\widetilde{\mathcal D}_{\zeta} = \bigcup_{i=1}^n \mathcal B_d(q_i,\eta_0(q_i,\tilde{q}_i)) \times \mathcal B_d(\tilde{q_i},\eta_0(q_i,\tilde{q}_i)).
\]
Notice that, by the uniqueness result~\eqref{eq:lambda_iff}, the function $\Lambda:\widetilde{\mathcal D}_{\zeta} \to \R^m$ is well defined and smooth. Moreover, by construction, the properties~\eqref{eq:Pi} and~\eqref{eq:non_tang} are satisfied with $\mathcal{D}$ replaced by~$\widetilde{\mathcal D}_\zeta$.

Let us now define
\[
\tilde \alpha(\zeta) = \min_{i \in \{1,\ldots,n\}} \alpha_0(q_i,\tilde{q}_i) > 0.
\]
Up to reducing $\widetilde{\mathcal D}_{\zeta}$ to a smaller open set (still containing $\mathcal E_\zeta$), one can assume that for all $(q,\tilde{q}) \in \widetilde{\mathcal D}_{\zeta}$, it holds $\Lambda(q,\tilde{q}) \in \mathcal B_m(0,\tilde \alpha(\zeta))$ (this is a simple consequence of the smoothness of $\Lambda$, and the fact that $\Lambda=0$ on $\mathcal E_\zeta$). The property~\eqref{eq:lambda_minoree__} then follows from~\eqref{eq:lambda_iff}. Let us emphasize that we do not indicate the dependency of $\Lambda$ on $\zeta$, since, by the property~\eqref{eq:lambda_iff}, for any $\zeta \neq \zeta'$, the two functions $\Lambda$ defined on $\widetilde{\mathcal D}_\zeta$ and $\widetilde{\mathcal D}_{\zeta'}$ coincide on $\widetilde{\mathcal D}_\zeta \cap \widetilde{\mathcal D}_{\zeta'}$.
\end{proof}

We are now in position to prove Proposition~\ref{prop:proj_well_defined_}.

\begin{proof}[Proof of Proposition~\ref{prop:proj_well_defined_}]
Let us define, for any $\zeta > 0$, the set
\begin{equation}\label{eq:Dzeta}
{\mathcal D}_\zeta = \left\{ (q,\tilde{q}) \in \mathcal M \times \R^d, \, (q,\tilde{q}) \in \widetilde{\mathcal D}_{\zeta} \right\} = \widetilde{\mathcal D}_{\zeta} \cap \left(\mathcal{M} \times \R^d\right),
\end{equation}
and
\[
\mathcal D_{\rm imp} = \bigcup_{\zeta > 0} {\mathcal D}_\zeta.
\]
The set ${\mathcal D}_{\rm imp}$ is open in $\mathcal M \times \R^d$ as the union of open sets in $\mathcal M \times \R^d$. Notice that by construction, ${\mathcal D}_{\rm imp}$ contains the set $\mathcal E$ defined in~\eqref{eq:E}, since $\mathcal E=\cup_{\zeta >0} \mathcal E_\zeta$.

Let us now consider $(q_0,\tilde{q}_0) \in {\mathcal D}_{\rm imp}$, and let $\zeta_0 > 0$ be such that $(q_0,\tilde{q}_0) \in {\mathcal D}_{\zeta_0}$. The Lagrange multiplier $\Lambda(q_0,\tilde{q}_0)$ is then defined by the result of Lemma~\ref{prop:proj_well_defined__}. Notice in particular that it does not depend on the choice of $\zeta $ such that $(q_0,\tilde{q}_0) \in {\mathcal D}_\zeta$. The property~\eqref{eq:lambda_minoree_} follows from~\eqref{eq:lambda_minoree__} by considering $\mathcal V_0=\mathcal D_{\zeta_0}$ and $\alpha_0=\tilde \alpha(\zeta_0)$.
\end{proof}

 Let us conclude this section with the proof of Proposition~\ref{lem:DDtilde}, which is essentially based on compactness arguments.

\begin{proof}[Proof of Proposition~\ref{lem:DDtilde}]
Let us prove that for any $\zeta > 0$, there exists $\tilde \delta(\zeta)>0$ such that for any $(q,\tilde q) \in \mathcal M \times \R^d$, if $\mathcal{G}_M(q,\tilde q)$ is invertible, $\| \mathcal{G}_M(q,\tilde q)^{-1}\| \le \zeta$, and $\|\xi(\tilde q)\| \le \tilde \delta(\zeta)$, then $(q,\tilde{q}) \in \mathcal D_\zeta$, where $\mathcal D_\zeta$ is defined by~\eqref{eq:Dzeta}. This obviously implies the first item of Proposition~\ref {lem:DDtilde} by considering
$$\delta(\zeta)=\sup_{\tilde \zeta > \zeta} \tilde\delta (\tilde \zeta).$$
We proceed by contradiction. Assume that, for some $\zeta > 0$, there exists a sequence $(q_n, \tilde{q}_n)_{n \ge 1}$ with values in $\mathcal M \times \R^d$ such that 
\[
\forall n \geq 1, \qquad \|\xi(\tilde{q}_n)\| \le \frac 1 n, \, \left\| \mathcal{G}_M(q_n,\tilde q_n)^{-1} \right\| \le \zeta \text{ and } (q_n, \tilde{q}_n) \not \in {\mathcal D}_\zeta.
\]
Then, by a compactness argument (recall~\eqref{eq:Mcompact}), and up to the extraction of a subsequence, one can assume that $\lim_{n \to \infty} (q_n,\tilde{q}_n) = (q,\tilde{q}) \in \mathcal M \times \R^d$. Moreover, by continuity, $\xi(\tilde{q})=0$, $\left\|\mathcal{G}_M(q,\tilde q)^{-1} \right\| \le \zeta$ and, since ${\mathcal D}_\zeta$ is an open subset of $\mathcal M \times \R^d$,  $(q, \tilde{q}) \not \in {\mathcal D}_\zeta$. This is in contradiction with the fact that $(q,\tilde{q}) \in \mathcal E_\zeta \subset {\mathcal D}_\zeta$.

To prove the first part of the second item, let us first recall that  for all $q \in \mathcal M$, it holds $(q,q) \in \mathcal D_{\rm imp}$  (remember the invertibility assumption on the matrix $G_M(q)$ introduced in~\eqref{eq:GM}, and the first item in Proposition~\ref{prop:proj_well_defined_}). Let us now argue by contradiction. Assume that, for any $n \ge 1$, there exist $q_n \in \mathcal M$ and $\tilde{q}_n \in \mathcal B_d(q_n, \frac 1 n)$ such that $(q_n, \tilde{q}_n) \not \in \mathcal D_{\rm imp}$. By a compactness argument, and upon extracting a subsequence, one can assume that $\lim_{n \to \infty} q_n= \lim_{n \to \infty} \tilde{q}_n=q_\infty \in \mathcal M$. Since $\mathcal D_{\rm imp}$ is open in $\mathcal M \times \R^d$, necessarily, $(q_\infty,q_\infty) \not \in \mathcal D_{\rm imp}$, which is in contradiction with the fact that $(q,q) \in \mathcal D_{\rm imp}$ for all $q \in \mathcal{M}$. This shows that there exists $\beta_0>0$ such that, for any $q \in \mathcal M$ and any $\tilde{q} \in \mathcal B_{d}(q,\beta_0)$, it holds $(q, \tilde{q}) \in \mathcal D_{\rm imp}$.

Now, for any $q \in \mathcal M$, one has $(q,q) \in \mathcal D_{\rm imp}$ and thus, from Proposition~\ref{prop:proj_well_defined_}, there exists a neighborhood $\mathcal V_0(q)$ of $(q,q)$ in $\mathcal M \times \R^d$ such that $\mathcal V_0(q) \subset \mathcal D_{\rm imp}$, together with an associated positive real number~$\alpha_0(q)$ such that~\eqref{eq:lambda_minoree_} holds. One can assume without loss of generality that $\mathcal V_0(q)=\{ (r,\tilde{r}) \in \mathcal M \times \R^d,  r \in  \mathcal B_d(q,\varepsilon(q)), \tilde{r} \in \mathcal B_d(r,\varepsilon(q))\}$ for some positive $\varepsilon(q)$. The set $\cup_{q \in \mathcal M} \mathcal V_0(q)$ is an open cover of the compact set $\mathcal M \times \mathcal M$ from which one can extract a finite cover:
$\mathcal M \times \mathcal M \subset \cup_{i=1}^n \mathcal V_0(q_i)$.  Let us introduce $\alpha_1= \min_{i=1, \ldots n} \alpha_0(q_i)$.
 Upon further reducing $\beta_0>0$, one can assume that for all $q \in \mathcal M$ and for all $\tilde q \in \mathcal B_d(q,\beta_0)$, it holds $(q, \tilde q) \in \cup_{i=1}^n \mathcal V_0(q_i)$ and $\|\Lambda(q,\tilde q)\| < \alpha_1$. The property~\eqref{eq:lambda_minoree} then follows from~\eqref{eq:lambda_minoree_}.
\end{proof}

\subsubsection{Example 1: Properties of the numerical scheme based on the implicit function theorem}\label{sec:rattle_imp}

From the set $\mathcal D_{\rm imp}$ introduced in Proposition~\ref{prop:proj_well_defined_}, and for a given timestep $\Delta t>0$, we can now define the associated set~$A_{\rm imp}$ by~\eqref{eq:def_A}:
\begin{equation}
  \label{eq:def_A_eps_}
  A_{\rm imp} = \left\{ (q,p) \in T^*\mathcal{M}, \, \ \left(q, q+ \Delta t \, M^{-1} \left[p - \frac{\Delta t}{2} \nabla V (q) \right] \right) \in \cD_{\rm imp}  \right\},
\end{equation}
and the associated set  $B_{\rm imp}$ by~\eqref{eq:B}:
\begin{equation}
  \label{eq:B_eps}
  B_{\rm imp}= \Big\{ (q,p) \in A_{\rm imp} \cap \Psi_{\Delta t}^{-1}(A_{\rm imp}), \ (\Psi_{\Delta t}\circ \Psi_{\Delta t})(q,p) = (q,p) \Big\}.
\end{equation}
The function $\Psi_{\Delta t}$ in~\eqref{eq:B_eps} is the one defined in Section~\ref{sec:Psi_dt}, with the Lagrange multiplier function defined by the implicit function theorem as introduced in the previous section.

\paragraph{On the sets $A_{\rm imp}$ and $B_{\rm imp}$.}
 The next lemma shows that the admissible Lagrange multiplier function introduced in the previous section provides a framework in which the set $B$ is non empty.

\begin{lem}
  \label{lem:A_B_non_empty}
  Assume that~\eqref{eq:Mcompact} holds, and let us consider the Lagrange multiplier function introduced in Section~\ref{sec:ex1_}. There exists $\Delta t_0 > 0$ such that, for all $\Delta t \in (0, \Delta t_0)$ and for all $q \in \mathcal{M}$, there is $p \in T^*_q \mathcal M$ for which $(q,p) \in A_{\rm imp}$ and $\Psi_{\Delta t}(q,p) = (q,p)$. In particular, $A_{\rm imp} \cap \Psi_{\Delta t}^{-1}(A_{\rm imp}) \neq \emptyset$ and $B_{\rm imp} \neq \emptyset$ for all $\Delta t \in (0, \Delta t_0)$.
\end{lem}
\begin{proof}
For any $q \in \mathcal M$, for $\Delta t$ sufficiently small, the momentum 
\begin{equation}\label{eq:pq}
 p=\frac{\Delta t}{2}\Pi_{T^*_{q} \mathcal M} (\nabla V(q))
\end{equation}
(where the projection $\Pi_{T^*_{q} \mathcal M}$ is defined in~\eqref{eq:proj_T*M}) is such that $(q,p) \in A_{\rm imp}$. Indeed, using the second item of Proposition~\ref{lem:DDtilde}, it is enough to prove that $\tilde{q} \in \mathcal B_d(q,\beta_0)$ where 
\begin{equation}\label{eq:qtildeq}
\tilde{q}=q+\Delta t M^{-1} \left( p - \frac{\Delta t}{2} \nabla V(q) \right)
\end{equation}
which is obviously the case for $\Delta t$ sufficiently small. Notice that the threshold on~$\Delta t$ can be chosen locally constant in~$q$.

Then, starting from the configuration $(q,p) \in T^* \mathcal M$ with $p$ given by~\eqref{eq:pq}, and following the RATTLE dynamics~\eqref{eq:RATTLE_num}, one obtains
\begin{align*}
q^1&=\tilde{q} + M^{-1} \nabla \xi(q) \Lambda\left(q, \tilde{q}\right),
\end{align*}
where 
\begin{align*}
\tilde{q} &= q+ \Delta t M^{-1} \left( p - \frac{\Delta t}{2} \nabla V(q) \right)= q -  \frac{\Delta t^2}{2} M^{-1} \nabla \xi(q) G_M^{-1}(q)[\nabla \xi(q)]^T M^{-1} \nabla V(q)\\
& =q- M^{-1} \nabla \xi(q) \Theta,
\end{align*}
with 
\begin{equation}\label{eq:qtheta}
\Theta=  \frac{\Delta t^2}{2} G_M^{-1}(q)[\nabla \xi(q)]^T M^{-1} \nabla V(q) \in \R^m.
\end{equation}
The equation $\xi( \tilde{q} + M^{-1} \nabla \xi(q) \lambda) = 
0$ thus admits two solutions: $\lambda=\Lambda(q,\tilde{q})$ and $\lambda=\Theta$. Moreover, up to reducing further  the threshold on $\Delta t$, one can assume that $\|\Theta\| < \alpha_1$, where $\alpha_1 >0$ has been introduced in the second item of Proposition~\ref{lem:DDtilde}.  Again, the threshold on~$\Delta t$ can be chosen locally constant in~$q$.
Therefore, from~\eqref{eq:lambda_minoree}, it follows that $\Lambda\left(q, \tilde{q}\right) = \Theta$. This implies that $p^{1/2}=0$ and $q^1=q$. It is then easy to check that $p^1=-p$. 

The above construction was performed for an arbitrary element $q \in \mathcal{M}$. By a compactness argument, one can define a limiting timestep $\Delta t_0>0$ such that, for all $\Delta t \in (0,\Delta t_0)$ and for all $q \in \mathcal M$, it holds $(q,\tilde q) \in \mathcal D_{\rm imp}$ and $\Theta=\Lambda(q,\tilde q)$, where $p$, $\tilde q$ and $\Theta$ are respectively defined by~\eqref{eq:pq},~\eqref{eq:qtildeq} and~\eqref{eq:qtheta}. Thus, for any $\Delta t \in  (0,\Delta t_0)$ and any $q\in \mathcal M$, we exhibited a momentum $p$ such that $(q,p) \in A_{\rm imp}$ and 
\[
\Psi_{\Delta t} (q,p) =(q,p).
\]
This already shows that $A_{\rm imp} \cap \Psi_{\Delta t}^{-1} (A_{\rm imp})$ is a non empty open set. Moreover, one obviously has 
\[
\left(\Psi_{\Delta t} \circ\Psi_{\Delta t}\right)(q,p) =(q,p),
\]
and thus, from Lemma~\ref{lem:B}, the set $B_{\rm imp}$ defined by~\eqref{eq:B} contains at least the connected component of $A_{\rm imp} \cap \Psi_{\Delta t}^{-1} (A_{\rm imp})$ to which $(q,p)$ belongs. The set $B_{\rm imp}$ is thus a non empty open set.
\end{proof}

\paragraph{Reversibility of the RATTLE scheme up to momentum reversal.}

The following lemma shows the local reversibility up to momentum reversal of the RATTLE scheme, for the admissible Lagrange multiplier function introduced in the previous section.

\begin{lem}
\label{lem:rev_Psi_dt}
Assume~\eqref{eq:Mcompact}, and consider the Lagrange multiplier function introduced in Section~\ref{sec:ex1_} as well as the constant $\beta_0$ introduced in Proposition~\ref{lem:DDtilde}. There exists $\beta_1 \in (0,\beta_0)$ (independent of $\Delta t>0$) such that, for any $(q,p) \in T^* \mathcal M$ for which 
\begin{equation}\label{eq:rattle_rev_1}
\left\|\Delta t M^{-1}\left(p - \frac{\Delta t}{2} \nabla V(q)\right)\right\| < \beta_1,
\end{equation}
and
\begin{equation}\label{eq:rattle_rev_2}
\left\|\Delta t M^{-1}\left(-p^1 - \frac{\Delta t}{2} \nabla V(q^1)\right)\right\| < \beta_1,
\end{equation}
where $(q^1,p^1)$ is the state obtained after one step of the RATTLE scheme~\eqref{eq:RATTLE_num} starting from $(q,p)$ (namely $(q^1,-p^1)=\Psi_{\Delta t}(q,p)$), the configurations $(q,p)$ and $(q^1,-p^1)$ belong to $A_{\rm imp}$ and 
$$ \left(\Psi_{\Delta t} \circ \Psi_{\Delta t}\right)(q,p)=(q,p).$$
\end{lem}
Note that, thanks to the second item of Proposition~\ref{lem:DDtilde} and the inequality $\beta_1 < \beta_0$, the condition~\eqref{eq:rattle_rev_1} ensures that $\Lambda\left(q,q+\Delta t M^{-1}\left(p - \frac{\Delta t}{2} \nabla V(q)\right)\right)$ is well defined, and thus that $\Psi_{\Delta t}(q,p)$ is well defined. Likewise, one can check that  the condition~\eqref{eq:rattle_rev_2} ensures that $\Psi_{\Delta t}(q^1,-p^1)=\Psi_{\Delta t}(\Psi_{\Delta t}(q,p))$ is well defined.

There are two ways to satisfy the assumptions~\eqref{eq:rattle_rev_1}--\eqref{eq:rattle_rev_2}. The first one is to fix a configuration $(q,p) \in T^*\mathcal{M}$ and to choose a sufficiently small timestep~$\Delta t$. The second one is to place an upper limit on the norm of the momenta: for any $R>0$, there exists $\Delta t_R > 0$ such that  for all $\Delta t \in (0, \Delta t_R)$ and for all $(q,p) \in T^* \mathcal M$ for which $\|p\| < R$, the assumptions~\eqref{eq:rattle_rev_1}--\eqref{eq:rattle_rev_2} are satisfied.

Lemma~\ref{lem:rev_Psi_dt} shows that, for the Lagrange multiplier function $\Lambda$ introduced in this section, if one considers the proposal move based on the subset of $A_{\rm imp}$ defined by
\[
\widetilde{A}_{\rm imp} = \Big\{ (q,p) \in T^* \mathcal M, \, \text{ \eqref{eq:rattle_rev_1} holds} \Big\},
\]
then Steps (4)-(5) in the reverse projection check (see Section~\ref{sec:psi_B}) are actually not necessary. Indeed, if \eqref{eq:rattle_rev_1}--\eqref{eq:rattle_rev_2} holds, then necessarily $(\Psi_{\Delta t} \circ \Psi_{\Delta t})(q,p) = (q,p)$. In other words, 
\[
\widetilde{B}_{\rm imp}= \Big\{ (q,p) \in \widetilde A_{\rm imp} \cap \Psi_{\Delta t}^{-1}(\widetilde A_{\rm imp}), \ (\Psi_{\Delta t}\circ \Psi_{\Delta t})(q,p) = (q,p) \Big\}= \widetilde A_{\rm imp} \cap \Psi_{\Delta t}^{-1}(\widetilde A_{\rm imp}).
\]

\begin{proof}[Proof of Lemma~\ref{lem:rev_Psi_dt}]
For a given configuration $(q,p) \in T^* \mathcal M$, the assumption~\eqref{eq:rattle_rev_1} together with  the second item of Proposition~\ref{lem:DDtilde} imply that 
\[
\left(q,q+\Delta t M^{-1}\left(p - \frac{\Delta t}{2} \nabla V(q)\right)\right) \in {\mathcal D}_{\rm imp}.
\]
It is then possible to define the state $(q^1,p^1)$ obtained after one step of the RATTLE scheme~\eqref{eq:RATTLE_num} starting from $(q,p)$:
$$(q^1,-p^1)=\Psi_{\Delta t}(q,p).$$
Similarly, under the assumption~\eqref{eq:rattle_rev_2}, one can define the state $(q^2,p^2)$ obtained after one step of the RATTLE scheme~\eqref{eq:RATTLE_num} starting from $\Psi_{\Delta t}(q,p)$:
$$(q^2,-p^2)=\Psi_{\Delta t}(q^1,-p^1).$$
In particular, $(q^2,-p^2)=(\Psi_{\Delta t} \circ \Psi_{\Delta t})(q,p)$. Our aim is to prove that $q^2=q$ and $p^2=-p$.

We have from~\eqref{eq:RATTLE_num}:
\begin{equation}\label{eq:q1p1_}
\left\{
\begin{aligned}
q^1&=q+\Delta t M^{-1} \left( p - \frac{\Delta t}{2} \nabla V(q) +
  \nabla \xi(q) \lambda^{1/2}\right),\\
p^1&=p - \frac{\Delta t}{2} (\nabla V(q) + \nabla V(q^1)) + \nabla
\xi(q) \lambda^{1/2} +\nabla \xi (q^1) \lambda^1,
\end{aligned}
\right.
\end{equation}
and
\begin{equation}\label{eq:q2p2_}
\left\{
\begin{aligned}
q^2&=q^1+\Delta t M^{-1} \left( -p^1 - \frac{\Delta t}{2} \nabla V(q^1) +
  \nabla \xi(q^1) \lambda^{3/2} \right),\\
p^2&=-p^1 - \frac{\Delta t}{2} (\nabla V(q^1) + \nabla V(q^2)) + \nabla
\xi(q^1) \lambda^{3/2} +\nabla \xi (q^2) \lambda^2.
\end{aligned}
\right.
\end{equation}
Equation~\eqref{eq:q1p1_} can be reformulated as
\begin{equation}
  \label{eq:q_q2_}
  q=q^1+\Delta t M^{-1}\left(-p^1 - \frac{\Delta t}{2} \nabla V(q^1) + \nabla \xi(q^1) \lambda^1 \right).
\end{equation}
Therefore, there are two solutions $\lambda=\Delta t \lambda^1$ and $\lambda=\Delta t \lambda^{3/2}$ to the equation
\[
\xi\left(q^1+\Delta t M^{-1}\left(-p^1 - \frac{\Delta t}{2} \nabla V(q^1)\right) + M^{-1} \nabla \xi(q^1) \lambda \right)=0.
\]
Let us prove that $\lambda^1 = \lambda^{3/2}$, which will then imply $q=q^2$ from~\eqref{eq:q2p2_}-\eqref{eq:q_q2_}. The second equation in~\eqref{eq:q1p1_} leads to
\begin{align*}
  \Delta t \lambda^{1} = \Delta t \, G_M^{-1}(q^1) [\nabla \xi(q^1)]^T M^{-1} \left( p^1 + \frac{\Delta t}{2} \nabla V(q^1) - p + \frac{\Delta t}{2} \nabla V(q) -\nabla \xi(q) \lambda^{1/2} \right),
\end{align*}
so that, since $\Delta t \lambda^{1/2}=\Lambda\left(q, q+\Delta t M^{-1}\left(p - \frac{\Delta t}{2} \nabla V(q)\right) \right)$,
\begin{align}
  \|\Delta t \lambda^{1}\| &\le \left\| G_M^{-1} \right\|_{L^\infty(\mathcal M)} \|\nabla \xi \|_{L^\infty(\mathcal M)} \left(2\beta_1 +\|M^{-1}\| \| \nabla \xi\|_{L^\infty(\mathcal M)} \| \Delta t \lambda^{1/2}\| \right) \nonumber \\
  &\le  \| G_M^{-1} \|_{L^\infty(\mathcal M)} \|\nabla \xi \|_{L^\infty(\mathcal M)} \left(2\beta_1 +\|M^{-1}\| \| \nabla \xi\|_{L^\infty(\mathcal M)} \sup_{q\in \mathcal M, \tilde q \in \mathcal B_d(q,\beta_1)} \left\|\nabla_{\tilde{q}}\Lambda(q, \tilde q)\right\| \beta_1 \right), \label{eq:estimlambda1}
\end{align}
where we used the assumptions~\eqref{eq:rattle_rev_1}--\eqref{eq:rattle_rev_2}. 
From the estimate~\eqref{eq:estimlambda1}, one can obviously choose $\beta_1 \in (0,\beta_0)$, not depending of $(q,p,\Delta t)$, such that, for all $(q,p)$ for which~\eqref{eq:rattle_rev_1}--\eqref{eq:rattle_rev_2} are satisfied, 
\begin{align}
  \|\Delta t \lambda^{1}\| &< \alpha_1,
 \label{eq:majo_}
\end{align}
where  $\alpha_1 >0$ has been introduced in the second item of Proposition~\ref{lem:DDtilde}. 
Since \linebreak $\Delta t \lambda^{3/2}=\Lambda\left(q^1, q^1+\Delta t M^{-1}\left(-p^1 - \frac{\Delta t}{2} \nabla V(q^1)\right) \right)$, the property~\eqref{eq:lambda_minoree} thus implies that $\lambda^1=\lambda^{3/2}$, and as a consequence $q=q^2$. 

Now, by considering the second equations in~\eqref{eq:q1p1_} and~\eqref{eq:q2p2_}, 
\begin{align*}
  p^2&=-p +\frac{\Delta t}{2} (\nabla V(q) - \nabla V(q^2)) - \nabla \xi(q) \lambda^{1/2} + \nabla \xi(q^1) (\lambda^{3/2}-\lambda^1) + \nabla \xi(q^2) \lambda^2\\ 
  &=-p + \nabla\xi(q) (\lambda^2-\lambda^{1/2}). 
\end{align*}
Since $p^2$ and $p$ are in $T^*_q{\mathcal M}$, it holds
\[
0 = [\nabla \xi(q)]^T M^{-1} (p^2+p) = G_M(q)(\lambda^2-\lambda^{1/2}),
\]
which implies $\lambda^2=\lambda^{1/2}$ and thus $p^2=-p$. This concludes the proof of the equality $\left(\Psi_{\Delta t} \circ \Psi_{\Delta t}\right)(q,p)=(q,p)$.
\end{proof}

Two further comments are in order here. First, the procedure we used in Section~\ref{sec:ex1_} to build the Lagrange multiplier function is theoretical in nature: it is based on the implicit function theorem, which does not provide an algorithm to actually compute the Lagrange multiplier (see however Remark~\ref{rem:simple_example} below). Second, this Lagrange multiplier function is by construction restricted to positions~$\tilde{q}$ living in some neighborhood of the submanifold $\mathcal M$, and this is typically satisfied only for small timesteps, but not for large ones. In practice, it is useful to use large timesteps, since 
this may allow to leave modes of the target distribution in fewer iterations (see Section~\ref{sec:numerics} for some numerical illustrations). 
Our objective now is thus to extend the definition of the Lagrange multiplier function so that it can be used even for large timesteps, and to make it implementable in practice. This is done in the next section, using Newton's method. The importance of Steps (4)-(5) in this context is illustrated numerically in Section~\ref{sec:numerics}.

\begin{rem}
  \label{rem:simple_example}
  In some cases, it is possible to identify the solution to the implicit function theorem and the associated domain $\mathcal D$ by solving explicitly the  equation in $\lambda$: $\xi(\tilde{q}+M^{-1} \nabla \xi(q) \lambda)=0$. In simple situations, it is then even possible to define $A$ by fixing an explicit threshold $R_{\Delta t}$ on the momenta (see~\cite[Algorithm~3.58]{lelievre-rousset-stoltz-book-10}), in which case $A=\{(q,p) \in T^* \mathcal M, \, \|p\|^2 < R_{\Delta t} \}$. As discussed after Lemma~\ref{lem:rev_Psi_dt}, Steps~(4)-(5) in the reverse projection check above are not needed in this case for sufficiently small~$\Delta t$: it is enough to simply check that $(q,p) \in A$ and $\Psi_{\Delta t}(q,p) \in A$.

  As a simple example, let us consider the parameters $d=2$, $M={\rm Id}$, $V=0$ and $\xi(q)=\|q\|^2-1$. Solving $(C_q)$ in~\eqref{eq:RATTLE} then amounts to finding the roots of a polynomial of order two in $\lambda^{n+1/2}$. Since $\left\|q^{n+1}\right\|^2 = \left\|q^n\right\|^2 = 1$ and $p^n \in T^*_{q^n}\mathcal{M}$, \emph{i.e.} $q^n \cdot p^n = 0$, it holds 
  \[
  \left(\lambda^{n+1/2}\right)^2 + \frac{1}{\Delta t} \lambda^{n+1/2} + \frac{\left\|p^n\right\|^2}{4} = 0. 
  \]
  This polynomial admits a solution if and only if $|p^n|^2 \le \Delta t^{-2}$. When this condition is satisfied, one chooses the Lagrange multiplier with the smallest absolute value to define $\lambda^{n+1/2}$:
  \[
  \lambda^{n+1/2} = \frac{-1+\sqrt{1 - \Delta t^2 \|p^n\|^2}}{2\Delta t}.
  \]
  It is then easy to check that $\lambda^{n+1}=\lambda^{n+1/2}$, and that, for $A=\{(q,p) \in T^* \mathcal M, \, \|p\|^2 < \Delta t^{-2} \}$, this defines a map $\Psi_{\Delta t}: A \to T^* \mathcal M$ will all the required properties. In this very simple case, one can actually check that $|p^{n+1}| = |p^n|$ with $(q^{n+1},p^{n+1}) = \Psi_{\Delta t}(q^n,p^n)$, so that $\Psi_{\Delta t} (A) = A$. It is therefore not even necessary to use Step (3) in the reverse projection check.
\end{rem}

\subsubsection{Example 2: A global and more practical projection based on Newton's method}
\label{sec:ex2}

A practical and general way to define the ensemble $A$ is to construct it as the set of configurations $(q^n,p^n)$ for which a Newton's algorithm yields a solution to $(C_q)$
(up to numerical precision) before a fixed given number of iterations
(see~\cite[Section~3.3.5.1]{lelievre-rousset-stoltz-book-10}
and~\cite{zappa-holmes-cerfon-goodman-17}). 

In order to put this practical algorithm in our framework, let us consider the following idealized Newton's algorithm to define the Lagrange multiplier function~$\Lambda_{\rm newt}$ and the associated admissible ensemble $\mathcal D_{\rm newt}$. For given $q \in \mathcal M$ and $\tilde{q} \in \R^d$, let us perform a given number $N_{\rm newt}$ of Newton's iterations to solve the nonlinear problem
\[
\text{find } \lambda \in \R^m \text{ such that }\xi(\tilde{q} + M^{-1} \nabla \xi(q) \lambda)=0,
\]
starting from a given value of $\lambda$ (say 0 to fix the ideas). The precise algorithm is the following.

\begin{algo}[Idealized Newton's algorithm]
  \label{algo:idealized_Newton}
Starting from $\theta^0=0$, one performs the following iterations: iterate on $n =0, \ldots, N_{\rm newt}$, 
\begin{enumerate}[(1)]
\item If $[\nabla \xi(\tilde{q} + M^{-1} \nabla \xi(q) \theta^n)]^T M^{-1} \nabla
    \xi(q)$ is not invertible then set $(q,\tilde{q}) \not \in \mathcal D_{\rm newt}$ and exit the loop;
\item Otherwise, compute 
\[
\theta^{n+1} = \theta^n - \left( [\nabla \xi(\tilde{q} + M^{-1} \nabla \xi(q) \theta^n)]^T M^{-1} \nabla \xi(q) \right)^{-1} \xi\left(\tilde{q} + M^{-1} \nabla \xi(q) \theta^n\right).
\]
Increment $n$ and go back to (1).
\end{enumerate}
If these iterations are performed successfully (namely if $n=N_{\rm newt}+1$ when exiting the loop), one then proceeds to the following steps, by considering the last configuration obtained after the Newton's iterations:
\[
\widehat{q}=\tilde{q} + M^{-1} \nabla \xi(q) \theta^{N_{\rm newt}}.
\]
\begin{enumerate}[(1)]
\setcounter{enumi}{2}
\item If $(q,\hat{q}) \not \in \mathcal D_{\rm imp}$ then set $(q,\tilde{q}) \not \in \mathcal D_{\rm newt}$ and exit.
\item Otherwise, set $(q,\tilde{q}) \in \mathcal D_{\rm newt}$ and $\Lambda_{\rm newt}(q,\tilde{q})=\Lambda(q,\widehat q)$, where $\Lambda$ is the Lagrange  multiplier function defined in Proposition~\ref{prop:proj_well_defined_}. 
\end{enumerate}
\end{algo}

This algorithm is in fact an idealized method, where the initial Newton's iterations allow to bring $\hat{q}$ sufficiently close to $\mathcal{M}$ in order for the projection to be uniquely defined by Proposition~\ref{prop:proj_well_defined_} (the initial position $\tilde{q}$ is indeed in general not close to $\mathcal{M}$ for large timesteps). Notice that thanks to the first item of Proposition~\ref{lem:DDtilde} and since $\mathcal{G}_M(q,\hat{q})$ is invertible by construction, if $\|\xi(\hat{q})\|$ is small (namely $\|\xi(\hat{q})\|<\delta(\|\mathcal{G}_M(q,\hat{q})^{-1}\|)$), then $(q,\hat{q}) \in \mathcal D_{\rm imp}$.

The following proposition is a direct consequence of the previous results.
\begin{prop}
Assume that~\eqref{eq:Mcompact} holds.  The function $\Lambda_{\rm newt}$ defines an admissible (in the sense of Definition~\ref{def:proj_sel}) Lagrange multiplier function on the open set~$\mathcal D_{\rm newt}$.
\end{prop}

Let us finally comment on a few practical aspects. In practice, one simply performs Newton's iterations, without the final projection steps (3) and (4). Moreover, the number of iterations is  limited by using a stopping criterion of the form $\|\xi(\tilde{q}+ M^{-1} \nabla \xi(q) \theta^{n+1}) \| < \varepsilon_{\rm newt}$ and/or $\|\theta^{n+1}-\theta^n\| < \varepsilon_{\rm newt} \|\theta^n\|$ for some tolerance $\varepsilon_{\rm newt}>0$. If these stopping criteria are not satisfied after $N_{\rm newt}$ iterations, the projection is unsuccessful, which implicitly defines $\mathcal D_{\rm newt}$. Notice that the projection is also unsuccessful if the matrix $[\nabla \xi(\tilde q + M^{-1} \nabla \xi(q) \theta^n)]^T M^{-1} \nabla \xi(q)$ happens to be numerically non invertible at some iteration $n$ (numerical non invertibility here means that the inverse is ill conditioned; see the discussion in Section~\ref{sec:summary_algo}). To connect to the idealized algorithm, one then needs to assume that, once the stopping criteria are satisfied, the position $\widehat{q}$ is so close to the submanifold $\mathcal M$ that the remaining steps of the idealized algorithm (namely the additional Newton's iterations, as well as the final projection step using Proposition~\ref{prop:proj_well_defined_}) will not modify the projected point significantly (namely, within numerical precision). This is reasonable for sufficiently small tolerances~$\varepsilon_{\rm newt}$. In practice, in the numerical experiments of Section~\ref{sec:numerics}, we use $N_{\rm newt}= 100$ and $\varepsilon_{\rm newt}= 10^{-12}$.

Note also that there are other ways of updating the Lagrange multipliers, for instance by successfully updating the components of~$\theta$ instead of performing a global update. This may be useful in situations when the dimensionality~$m$ of the constraint is large; see~\cite[Section~7.2]{leimkuhler-reich-04} for further precisions. It should be possible to adapt our framework to these types of algorithms.

\subsection{Constrained Generalized Hybrid Monte Carlo schemes}
\label{sec:GHMC_schemes}

Once it is clear how to integrate the constrained Hamiltonian dynamics in a reversible way, Hybrid Monte Carlo (HMC) and Generalized Hybrid Monte Carlo (GHMC) schemes to sample $\mu$ are readily obtained, as made precise in this section. 


\subsubsection{Constrained Hybrid Monte Carlo scheme}
\label{sec:constrained_HMC}

We start by presenting a HMC scheme, which corresponds to integrating one step of the Hamiltonian dynamics, before resampling the momenta in the tangent space. The one-step HMC scheme we consider here can actually be seen as a time-discretization with a timestep $\Delta t^2/2$ of the constrained overdamped Langevin dynamics (see~\cite[Remark~3.7]{lelievre-rousset-stoltz-12}), which reads:
\[
\left\{ \begin{aligned}
dq_t & = -M^{-1}\nabla V(q_t) \, dt + \sqrt{2}M^{-1/2} dW_t + M^{-1}\nabla\xi(q_t) \, d\lambda_t, \\
\xi(q_t) & = 0.
\end{aligned} \right.
\]
This is why we call  Metropolized Adjusted Langevin Algorithm (MALA) the corresponding numerical scheme in Section~\ref{sec:numerics}.

\begin{algo}[Constrained Hybrid Monte Carlo]
\label{algo:HMC}
Starting from $(q^n,p^n) \in T^*\mathcal M$,
\begin{enumerate}[(i)]
\item Sample $\tilde{p}^{n}$ in $T_{q^n}^* \mathcal M$ according to the distribution $\kappa_{q^n}$ defined in~\eqref{eq:kappa};
\item Use one step of the RATTLE dynamics with momentum reversal and reverse projection check: 
  \[
  (\tilde{q}^{n+1},\tilde p^{n+1})=\Psi_{\Delta t}^{\rm rev}(q^n,\tilde p^n);
  \]
\item Draw a random variable $U^n$ with uniform law on $(0,1)$:
\begin{itemize}
\item if $U^n \le \exp(-H(\tilde q^{n+1},\tilde p^{n+1})+H(q^n,p^n))$, accept the proposal: $({q}^{n+1},p^{n+1})=(\tilde{q}^{n+1},\tilde p^{n+1})$;
\item otherwise reject the proposal: $({q}^{n+1},p^{n+1})=(q^n,\tilde p^n)$.
\end{itemize}
\end{enumerate}
\end{algo}

For a discussion on the practical way to implement Step~(i), we refer to Remark~\ref{rem:sample_p} below. By construction and thanks to the properties of $\Psi_{\Delta t}^{\rm rev}$ (see Proposition~\ref{prop:PsiA}), the Markov chain $(q^n,p^n)_{n \ge 0}$ admits $\mu$ as an invariant measure, as the composition of two steps which are reversible\footnote{Notice that the composition of two Markov kernels which are $\pi$-reversible is not $\pi$-reversible in general; it however admits $\pi$ as an invariant measure.} with respect to $\mu$: the sampling of independent and identically distributed (i.i.d.) momenta in Step (i), and the Metropolis--Hastings steps (Steps (ii)-(iii)). The reversibility of (ii)-(iii) relies on the following lemma whose proof is very similar to the unconstrained case (see~\cite{mehlig-heermann-forrest-92} and~\cite[Section~2.1.4]{lelievre-rousset-stoltz-book-10}).

\begin{lem}
The Metropolis--Hastings ratio associated with the Steps (ii)-(iii) of Algorithm~\ref{algo:HMC}, for a proposed move from $(q,p)$ to $(q',p')$, reads
\[
\frac{\delta_{\Psi_{\Delta t}^{\rm rev}(q',p')}(dq\,dp) \exp(-H(q',p')) \sigma_{T^*\mathcal M} (dq'\,dp')}{\delta_{\Psi_{\Delta t}^{\rm rev}(q,p)}(dq'\,dp') \exp(-H(q,p)) \sigma_{T^*\mathcal M} (dq\,  dp)}=\exp\left(-H(q',p')+H(q,p)\right).
\]
\end{lem}

\begin{proof}
The equality stated in the lemma is equivalent to
\[
\delta_{\Psi_{\Delta t}^{\rm rev}(q',p')}(dq\,dp) \, \sigma_{T^*\mathcal M} (dq'\,dp')=\delta_{\Psi_{\Delta t}^{\rm rev}(q,p)}(dq',dp') \, \sigma_{T^*\mathcal M} (dq\, dp).
\]
The proof of the latter identity is as follows. For a test function $\varphi(q,p,q',p')$, one has:
\begin{align*}
&\int_{(q,p) \in T^* \mathcal M}\int_{(q',p') \in T^* \mathcal M} \varphi (q,p,q',p') \delta_{\Psi_{\Delta t}^{\rm rev}(q',p')}(dq\,dp) \, \sigma_{T^*\mathcal M} (dq'\, dp')\\
& \qquad = \int_{(q',p') \in T^* \mathcal M} \varphi \left(\Psi_{\Delta t}^{\rm rev}(q',p'),q',p'\right) \, \sigma_{T^*\mathcal M} (dq'\, dp')\\
& \qquad = \int_{(q,p) \in B} \varphi (q,p,\left(\Psi_{\Delta t}^{\rm rev}\right)^{-1}(q,p)) \, \sigma_{T^*\mathcal M}(dq\,dp) + \int_{(q,p) \in B^c} \varphi (q,p,\left(\Psi_{\Delta t}^{\rm rev}\right)^{-1}(q,p)) \, \sigma_{T^*\mathcal M}(dq\,dp)\\
& \qquad = \int_{(q,p) \in T^* \mathcal M} \varphi (q,p,\Psi_{\Delta t}^{\rm rev}(q,p)) \, \sigma_{T^*\mathcal M}(dq\,dp)\\
& \qquad = \int_{(q,p) \in T^* \mathcal M}\int_{(q',p') \in T^* \mathcal M} \varphi (q,p,q',p') \delta_{\Psi_{\Delta t}^{\rm rev}(q,p)}(dq'\,dp') \, \sigma_{T^*\mathcal M}(dq\, dp),
\end{align*}
where we used the change of variable $(q,p)=\Psi_{\Delta t}^{\rm rev}(q',p')$ on $B$ and $B^c$ in the third line and the fact that $\left(\Psi_{\Delta t}^{\rm rev}\right)^{-1}=\Psi_{\Delta t}^{\rm rev}$ on both $B$ and $B^c$ in the fourth line, see Proposition~\ref{prop:PsiA}.
\end{proof}

To prove ergodicity, one then has to check irreducibility, see for example~\cite{schuette-98,cances-le-goll-stoltz-07,LBBG16,DMS17} for results in the unconstrained case, which can easily be adapted to constrained dynamics as in~\cite{hartmann-08}.

Let us emphasize that the invariant measure of the Markov chain $(q_n,p_n)_{n \ge 0}$ is $\mu$ whatever the potential function $V$ which is used to evolve the system in one step of the RATTLE scheme (Step~(ii) of Algorithm~\ref{algo:HMC}). The reversibility property is only based on the fact that $\Psi_{\Delta t}^{\rm rev}$ is an involution which preserves the phase space measure $\sigma_{T^* \mathcal M}(dq \,dp)$. In particular, one can simply use $V=0$ in RATTLE, which corresponds to the (constrained) random walk Metropolis Hastings algorithm considered in~\cite{zappa-holmes-cerfon-goodman-17}.

The interest of using the potential $V$ in the Hamiltonian dynamics is that, thanks to the good energy conservation of the RATTLE integrator, it is possible to run the Hamiltonian dynamics over large times while conserving the energy, and thus keeping a small rejection rate~\cite[Sections IX.5 and IX.8]{hairer-lubich-wanner-06}. In practice, it is thus possible to use in Step~(ii) of Algorithm~\ref{algo:HMC} several iterations of the RATTLE scheme (which corresponds to considering $(\tilde{q}^{n+1},\tilde{p}^{n+1})=(\Psi_{\Delta t}^{\rm rev})^K(q^n,\tilde p^n)$ for some integer $K$). There are thus three numerical parameters to choose  in practice: $M$, $\Delta t$ and $K$. The larger $\Delta t$ and $K$, the less correlated the samples, but the larger the rejection rate.

\begin{rem}
  \label{rem:truncated_momenta}
  In order to increase the probability that the forward and reverse projection checks are successful, it may be useful in practice to sample $\tilde p^n$ according to a truncated Gaussian random variable in Step~(i) of the HMC algorithm:
  \begin{equation}
    \label{eq:kappa_truncated}
    \tilde \kappa_q(dp)=\tilde{Z}^{-1} \exp \left(-\frac{p^T M^{-1}p}{2}\right) 1_{\{|p|^2 \le R_{\Delta t}\}}\sigma^{M^{-1}}_{T^*_q \mathcal M}(dp),
  \end{equation}
  for a well chosen $R_{\Delta t}$ (see also Remark~\ref{rem:simple_example} for related discussions). This can easily be done in practice by a rejection procedure. This leads to a dynamics which admits $ \tilde \mu(dq \, dp)=\nu(dq) \, \tilde\kappa_q(dp)$ as an invariant measure. Notice that  $\tilde \mu$ still has $\nu$ as a marginal in $q$.  Of course, a similar discussion holds for any distribution $\tilde\kappa_q(dp)$ on the set $|p|^2 \le R_{\Delta t}$ which is invariant under momentum reversal.
\end{rem}

\begin{rem}[Extra computational cost arising from the reverse projection check]
  In order to perform the reverse projection check, the forces at the proposed position need to be computed. This induces indeed some extra computational cost, which is however mitigated by the following facts: 
\begin{enumerate}[(i)]
\item it is necessary in any case to compute the energy at the proposed position in order to perform the Metropolis acceptance/rejection procedure. In many codes, the computation of the forces and the potential energy are done simultaneously.
\item it is possible to store the forces which are computed at the proposed position. They can be reused if the proposal is accepted. Therefore, additional force computations are performed only in situations when the forward Newton algorithm is successful but rejection occurs later on. In the numerical examples we provide in Section~\ref{sec:numerics} (see in particular the first lines of Table~\ref{tab:rejection}), rejections due to the reverse projection check (either Newton algorithm not well posed when starting from the proposed position or not returning to the initial configuration) account for at most 25\% of the rejections, most of the rejections arising either from the forward Newton algorithm not being successful or from the Metropolis step, depending on the timestep. 
\item let us also recall that simplified force fields can be used in the RATTLE proposals since the Metropolis step corrects for any bias arising from this simplification.
\end{enumerate}
We therefore believe that the extra computational cost should be quite small compared to standard HMC implementations, at least in situations where the number of constraints~$m$ is small compared to the dimension~$d$.
\end{rem}

\subsubsection{The constrained Generalized Hybrid Monte Carlo method}

Following~\cite{horowitz-91}, it is possible to derive variants of the constrained HMC, which are related to discretizations of the underlying constrained Langevin dynamics (see~\cite[Section 2.2.3]{lelievre-rousset-stoltz-book-10} for a general introduction):
\begin{equation}
  \label{eq:Langevin}
  \left\{
  \begin{aligned}
    dq_t& = M^{-1} p_t \, dt \,  ,\\
    dp_t&= -\nabla V(q_t) \, dt - \gamma M^{-1} p_t \, dt + \sqrt{2 \gamma}\,dW_t + \nabla \xi(q_t) \, d \lambda_t \, , \\
    \xi(q_t)&=0 \,,
  \end{aligned}
  \right.
\end{equation}
where $\gamma>0$ is the friction parameter (actually $\gamma$ can be chosen to be a symmetric positive definite matrix), and $(W_t)_{t \ge 0}$ is a standard $d$-dimensional Brownian motion.

The Generalized Hybrid Monte Carlo (GHMC) algorithm is obtained by a splitting technique, considering separately the constrained Ornstein-Uhlenbeck dynamics on $p_t$ on the one hand, and the constrained Hamiltonian dynamics on the other. A typical algorithm based on a Strang splitting and which is a variant of~\cite[Algorithm~3.56]{lelievre-rousset-stoltz-book-10} writes as follows (the sequences $(G^n)_{n \ge 0}$ and $(G^{n+1/2})_{n \ge 0}$ are independent sequences of independent Gaussian vectors in $\R^d$ with identity covariance). See also Section~\ref{sec:summary_algo} for a practical version written in pseudo-code.

\begin{algo}[Constrained Generalized Hybrid Monte Carlo]
  \label{algo:GHMC}
  Starting from $(q^n,p^n) \in T^*\mathcal M$,
  \begin{enumerate}[(i)]
  \item  Evolve the momenta according to the mid-point Euler scheme (with time increment $\Delta t/2$)
    \[
    \left\{
    \begin{aligned}
      &p^{n+1/4} = p^{n} - \frac{\Delta t}{4} \, \gamma \, M^{-1} \left(p^{n} + p^{n+1/4} \right) + \sqrt{\gamma\dt} \, G^{n} + \nabla \xi (q^n) \, \lambda^{n+1/4} \, , \\ 
      &\left[\nabla\xi(q^n)\right]^T M^{-1} p^{n+1/4} = 0 \, .
    \end{aligned}
    \right.
    \]
  \item  Apply one step of the RATTLE dynamics with momentum reversal and reverse projection check: 
    \[
    (\tilde{q}^{n+1},\tilde p^{n+3/4})=\Psi_{\Delta t}^{\rm rev}(q^n,p^{n+1/4}).
    \]
  \item Draw a random variable $U^n$ with uniform law on $(0,1)$:
    \begin{itemize} 
    \item if $U^n \le \exp(-H(\tilde q^{n+1},\tilde p^{n+3/4})+H(q^n,p^{n+1/4}))$, accept the proposal: $({q}^{n+1},p^{n+3/4})=(\tilde{q}^{n+1},\tilde p^{n+3/4})$;
    \item else reject the proposal: $({q}^{n+1},p^{n+3/4})=(q^n, p^{n+1/4})$.
    \end{itemize}
  \item Reverse momenta $\tilde p^{n+1}=-p^{n+3/4}$.
  \item Evolve the momenta according to the mid-point Euler scheme (with time increment $\Delta t/2$)
    \[
    \left\{
    \begin{aligned} 
      &p^{n+1} =  \tilde p^{n+1} - \frac{\Delta t}{4} \gamma M^{-1} ( \tilde p^{n+1} + p^{n+1} ) + \sqrt{\gamma\dt} \, G^{n+1/2} + \nabla \xi (q^{n+1}) \lambda^{n+1} \, ,\\
      &\nabla\xi(q^{n+1})^T M^{-1} p^{n+1} = 0 \, .
    \end{aligned}
    \right. 
    \]
\end{enumerate}
\end{algo}

The Markov chain $(q^n,p^n)_{n \ge 0}$ admits $\mu$ as an invariant measure, as the composition of four steps which are reversible with respect to the measure $\mu$: (i), (ii)-(iii), (iv) and (v). Notice that when $\gamma$ is chosen such that $\frac{\Delta t \gamma M^{-1}}{4} = \rm Id$, the constrained GHMC algorithm reduces to constrained HMC (Algorithm~\ref{algo:HMC}).

In the case when the forward and reverse projection checks are successful (Step (ii)), one can check that this algorithm is a consistent discretization of the constrained Langevin dynamics~\eqref{eq:Langevin}, both in the weak and strong sense (by an analysis similar to~\cite{BV09} for strong convergence). The discretization is of weak order~2 since the rejection rate in the Metropolis step is of order~$\Delta t^3$ (by an analysis similar to the one performed in~\cite{ST18}). Note the importance of Step~(iv) to get this property. This is why this algorithm may be interesting in practice: it also yields some dynamical information, while sampling exactly the stationary measure for~\eqref{eq:Langevin}. In the case when the proposal is rejected, Step (iv) implies that momenta are reversed, so that the trajectory ``goes backward'' in this case; see~\cite[Remark 2.13]{lelievre-rousset-stoltz-book-10} for further discussions, and~\cite{ST18} for an analysis of the weak error of the method.

\begin{rem}
  \label{rem:sample_p}
  As explained above, the implementation of Steps (i) and (v) amounts to a projection in an affine space, and can thus be performed exactly. Let us make this precise, for example for Step (i). In practice, one first computes the predicted unconstrained momentum
  \[
  \tilde{p}^{n+1/4}=\left( {\rm Id} + \frac{\Delta t}{4} \gamma M^{-1}\right)^{-1} \left( \left( {\rm Id} - \frac{\Delta t}{4} \gamma M^{-1}\right)  p^{n} + \sqrt{\gamma\dt} \, G^{n}\right).
  \]
  The Lagrange multiplier $\lambda^{n+1/4}$ is then obtained by solving the following linear system:
  \[
  \left[\nabla \xi(q^n)\right]^T M^{-1} \tilde{p}^{n+1/4} + \left[\nabla \xi(q^n)\right]^T M^{-1}\left( {\rm Id} + \frac{\Delta t}{4} \gamma M^{-1}\right)^{-1} \nabla \xi(q^n) \lambda^{n+1/4} = 0.
  \]
  The invertibility of the matrix $\left[\nabla \xi(q^n)\right]^T M^{-1} \left( {\rm Id} + \frac{\Delta t}{4} \gamma M^{-1}\right)^{-1} \nabla \xi(q^n)$ is a consequence of the invertibility of the matrix $G_M(q^n)$ (see~\eqref{eq:GM}).
  
  By choosing $\gamma$ such that $\frac{\Delta t \gamma M^{-1}}{4} = \rm Id$, this also gives a practical manner to implement Step~(i) in the constrained HMC algorithm~\ref{algo:HMC}.
\end{rem}

\begin{rem}
  Similar ideas can be used to take into account inequality constraints. More precisely, if $\mathcal M=\{q \in \R^d, \xi(q)=0 \text{ and } h(q) \le 0\}$ for some smooth function $h: \R^d \to \R^n$, there are two ways to take into account the additional constraints $h(q) \le 0$:
  \begin{itemize}
  \item simply include a rejection step in the algorithms (in Step (iii) of Algorithm~\ref{algo:HMC} and Algorithm~\ref{algo:GHMC}).
  \item implement a version of RATTLE dynamics with reflection on the boundary $h(q)=0$: this preserves the two fundamental properties of the scheme (measure preservation and involution) and thus gives an algorithm to sample the target distribution, see~\cite{kaufman-12,afshar-domke-15}.
  \end{itemize}
\end{rem}

\section{Numerical simulations}
\label{sec:numerics}

We illustrate in this section the importance of the reverse projection check in order to obtain an unbiased sampling, on a simple problem already considered in~\cite{zappa-holmes-cerfon-goodman-17}. We also discuss the relative efficiency of the algorithms introduced above, comparing in particular the constrained HMC and GHMC approaches. The section starts with a summary of the proposed algorithm in pseudo-code.

\subsection{Summary of the algorithm}
\label{sec:summary_algo}

The complete algorithm is summarized in the pseudo-codes below, see Numerical algorithms~\ref{algo:global}, \ref{algo:momentum_Lagrange_OU}, \ref{algo:momentum_Lagrange_RATTLE} and~\ref{algo:Newton}. The sampling algorithm consists in iterating procedure {\tt ConstrainedGHMC} of the algorithm (Numerical algorithm~\ref{algo:global}), which uses the procedures {\tt LAGRANGE$\_$MOMENTUM$\_$OU} (Numerical algorithm~\ref{algo:momentum_Lagrange_OU}) and {\tt LAGRANGE$\_$MOMENTUM$\_$RATTLE} (Numerical algorithm~\ref{algo:momentum_Lagrange_RATTLE}) to compute the Lagrange mutliplier for momentum constraints in the fluctuation/dissipation and RATTLE steps respectively, and {\tt NEWTON} (Numerical algorithm~\ref{algo:Newton}) to compute the Lagrange multiplier for position constraints. Some additional comments are in order:
\begin{itemize}
\item The algorithm decides to either make a specific choice of a Lagrange multiplier for position constraints, or to reject the possibility of performing a constrained RATTLE scheme. This is performed in practice by checking whether the practical Newton algorithm {\tt NEWTON} returns success starting from the unconstrained update of the position in the RATTLE scheme (see other remarks below for more details). Within our theoretical framework, this is equivalent to checking whether an initial configuration $(q,p)$ belongs to the set~$A$ defined by~\eqref{eq:def_A}.
\item The reversibility check is done by comparing the original position and the position obtained by performing a successful RATTLE step, followed by a momentum inversion, and an additional successful RATTLE step. Comparing the momenta is not necessary since the proof of Lemma~\ref{lem:B} shows that the configurations $(\hat{q},-\hat{p}) = (\Psi_{\Delta t}\circ\Psi_{\Delta t})(q,p)$ and~$(q,p)$ are different if and only if $\hat{q}\neq q$. The latter condition is implemented by checking whether $\|\hat{q}-q\| < \eta_{\rm rev}$ for some tolerance~$\eta_{\rm rev}$.
\item The matrix~$S$ is always invertible in Numerical algorithm~\ref{algo:momentum_Lagrange_OU} as discussed in Remark~\ref{rem:sample_p}. Numerical algorithms~\ref{algo:momentum_Lagrange_OU} and~\ref{algo:momentum_Lagrange_RATTLE} differ by a multiplication by $\left( {\rm Id} + \Delta t \gamma M^{-1}/4\right)^{-1}$, which arises from the specific choice of the discretization of the fluctuation/dissipation part in Algorithm~\ref{algo:GHMC}.
\item The Newton algorithm (Numerical algorithm~\ref{algo:Newton}) returns a boolean telling whether the projection of~$\tilde{q}$ along the direction $M^{-1}\nabla \xi(q)$ was successful, and a value of the parameter~$\theta$ such that $\xi(\tilde{q}+M^{-1}\nabla \xi(q)\theta) = 0$ (up to the tolerance $\eta_{\rm rev}$). 

  Note that a linear system has to be solved at each iteration. This can be done by computing the inverse of the matrix $[\nabla \xi(\tilde{q} + M^{-1} \nabla \xi(q) \theta)]^T M^{-1} \nabla \xi(q)$ as written in the pseudo-code. In this case, one should first check whether this matrix is indeed numerically invertible by checking for instance that its condition number (the relative range of its singular values) is not too large. Alternatively, one could solve the associated linear system using iterative methods. If the solution is not well posed, the algorithm exits the loop and returns {\tt FALSE} for the boolean indicating the success of the procedure.

  The user may want to consider a not too large maximal value $N_{\rm newt}$ for the number of inner loops in the Newton algorithm, especially for high-dimensional constraints for which the inversion of the matrix scales as $\mathrm{O}(m^3)$. When the Newton part of the algorithm is significant, it may be better to give up more easily, and possibly to use a larger tolerance $\eps_{\rm newt}$, at the cost of course of numerical biases. As discussed at the end of Section~\ref{sec:ex2}, it may be beneficial in those situations to consider projection procedures which do not require inverting a matrix or solving a linear system, as mentioned in~\cite[Section~7.2]{leimkuhler-reich-04}.

  Finally, the considered convergence criterion ensures that the value of~$\xi$ for the projected position is smaller than the convergence criterion $\eps_{\rm newt}$, and also that the difference between the iterates for the predicted projections on the submanifold are sufficiently close. It is of course possible to normalize differently the convergence criterion, for instance by working with relative errors on the positions, or by weighting the factors defining the criterion using different positive multiplicative constants.
\end{itemize}
Finally, let us stress again that the main result of the present paper (the reversibility of the proposed schemes) does not hold {\em exactly} for the implemented method because projections onto the sub-manifold are only performed in practice up to some numerical tolerances ($\eta_{\rm rev}$ and $\varepsilon_{\rm newt}$). The idealized case of the Newton algorithm with exact projection ($\varepsilon_{\rm newt} = 0$) is made precise in Algorithm~\ref{algo:idealized_Newton} above (see Section~\ref{sec:ex2}), and indeed gives rise to exactly reversible schemes.

\begin{algorithm}
\caption{\label{algo:global} One step of the practical constrained HMC algorithm with reverse projection}
\begin{algorithmic}
  \State Parameters: $\gamma$ (friction), $\Delta t$ (timestep), $\eta_{\rm rev}$ (tolerance for reverse check)
  \Procedure{ConstrainedGHMC}{$q$,$p$}
  \State $G \sim \mathcal{N}(0,\mathrm{Id})$
  \State $p \leftarrow ( {\rm Id} + \Delta t \gamma M^{-1}/4)^{-1} \left[ ( {\rm Id} - \Delta t \gamma M^{-1}/4)  p + \sqrt{\gamma\dt} \, G\right]$
  \State $\lambda = {\tt LAGRANGE\_MOMENTUM\_OU}(q,p)$
  \State $p \leftarrow p + ( {\rm Id} + \Delta t \gamma M^{-1}/4)^{-1} \nabla \xi(q) \lambda$
  \Comment{Integration of the fluctuation/dissipation for $\Delta t/2$}
  \State ${\tt Reject} = {\tt TRUE}$
  \State $\tilde{p} =  p - \Delta t \nabla V (q)/2$ and $\tilde{q} =  q + \Delta t \, M^{-1} \, \tilde{p}$
  \State Compute $({\tt Success\_forward\_RATTLE}, \theta) = {\tt NEWTON}(q,\tilde{q})$
  \If{${\tt Success\_forward\_RATTLE}$}
  \State $\tilde{p} \leftarrow \tilde{p} + \nabla \xi(q) \theta/\Delta t$ and $\tilde{q} \leftarrow \tilde{q} + M^{-1}\nabla \xi(q) \theta$
  \State $\tilde{p} \leftarrow \tilde{p} - \Delta t \nabla V(\tilde{q})/2$
  \State $\lambda = {\tt LAGRANGE\_MOMENTUM\_RATTLE}(\tilde{q},\tilde{p})$
  \State $\tilde{p} \leftarrow \tilde{p} + \nabla \xi(\tilde{q}) \lambda$
  \Comment{Constrained RATTLE -- proposition}
  \State $\hat{p} = -\tilde{p}$ and $\hat{q} = \tilde{q}$
  \State $\hat{p} \leftarrow \hat{p} - \Delta t \nabla V (\hat{q})/2$ and $\hat{q} \leftarrow  \hat{q} + \Delta t \, M^{-1} \, \hat{p}$
  \State Compute $({\tt Success\_backward\_RATTLE}, \theta) = {\tt NEWTON}(\tilde{q},\hat{q})$ 
  \If{${\tt Success\_backward\_RATTLE}$}
  \State $\hat{p} \leftarrow \hat{p} + \nabla \xi(\tilde{q}) \theta/\Delta t$ and $\hat{q} \leftarrow \hat{q} + M^{-1}\nabla \xi(\tilde{q}) \theta$
  \State $\hat{p} \leftarrow \hat{p} - \Delta t \nabla V(\hat{q})/2$
  \State $\lambda = {\tt LAGRANGE\_MOMENTUM\_RATTLE}(\hat{q},\hat{p})$
  \State $\hat{p} \leftarrow \hat{p} + \nabla \xi(\hat{q}) \lambda$
  \Comment{Constrained RATTLE -- reverse move}
  \If{$\| \hat{q}-q\| < \eta_{\rm rev}$}
  \Comment{Constrained RATTLE -- checking reversibility}
  \State $U \sim \mathcal{U}([0,1])$
  \State $\Delta H = H(\tilde{q},\tilde{p}) - H(q,p)$
  \If{$\log(U) \leq -\Delta H$}
  \Comment{Constrained RATTLE -- Metropolis acceptance/rejection}
  \State ${\tt Reject} = {\tt FALSE}$
  \EndIf
  \EndIf
  \EndIf
  \EndIf
  \If{${\tt Reject}$}
  \State $\tilde{p} = -p$ and $\tilde{q} = q$ 
  \EndIf
  \State $\tilde{G} \sim \mathcal{N}(0,\mathrm{Id})$
  \State $\tilde{p} \leftarrow ( {\rm Id} + \Delta t \gamma M^{-1}/4)^{-1} \left[ ( {\rm Id} - \Delta t \gamma M^{-1}/4) \tilde{p} + \sqrt{\gamma\dt} \, \tilde{G}\right]$
  \State $\lambda = {\tt LAGRANGE\_MOMENTUM\_OU}(\tilde{q},\tilde{p})$
  \State $\tilde{p} \leftarrow \tilde{p} + ( {\rm Id} + \Delta t \gamma M^{-1}/4)^{-1} \nabla \xi(\tilde{q}) \lambda$
  \Comment{Integration of the fluctuation/dissipation for $\Delta t/2$}
  \State {\bf return} $\tilde{q},\tilde{p}$
  \EndProcedure
\end{algorithmic}
\end{algorithm}

\begin{algorithm}
\caption{\label{algo:momentum_Lagrange_OU} Computation of the Lagrange multiplier for momentum constraints in OU part}
\begin{algorithmic}
  \Procedure{LAGRANGE$\_$MOMENTUM$\_$OU}{$q$,$p$}
  \State $S = \left[\nabla \xi(q)\right]^T M^{-1}\left( {\rm Id} + \Delta t \gamma M^{-1}/4\right)^{-1} \nabla \xi(q)$
  \State $b = \left[\nabla \xi(q)\right]^T M^{-1} p$
  \State {\bf return} $\lambda = -S^{-1}b$
  \EndProcedure
\end{algorithmic}
\end{algorithm}

\begin{algorithm}
\caption{\label{algo:momentum_Lagrange_RATTLE} Computation of the Lagrange multiplier for momentum constraints in RATTLE part}
\begin{algorithmic}
  \Procedure{LAGRANGE$\_$MOMENTUM$\_$RATTLE}{$q$,$p$}
  \State $S = \left[\nabla \xi(q)\right]^T M^{-1} \nabla \xi(q)$
  \State $b = \left[\nabla \xi(q)\right]^T M^{-1} p$
  \State {\bf return} $\lambda = -S^{-1}b$
  \EndProcedure
\end{algorithmic}
\end{algorithm}

\begin{algorithm}
\caption{\label{algo:Newton} Newton algorithm}
\begin{algorithmic}
  \State Parameters: $N_{\rm newt}$ (maximal number of iterations), $\varepsilon_{\rm newt}$ (tolerance for convergence)
  \Procedure{NEWTON}{$q$,$\tilde{q}$}
  \State Set $k = 0$, $\eta = 2 \varepsilon_{\rm newt}$, $\theta = 0$, $\theta_{\rm old} = 0$
  \While{($k \leq N_{\rm newt}$ {\bf and} $\eta > \varepsilon_{\rm newt}$)}
  \State $k \leftarrow k+1$
  \State $\theta_{\rm old} \leftarrow \theta$
  \If{$[\nabla \xi(\tilde{q} + M^{-1} \nabla \xi(q) \theta)]^T M^{-1} \nabla \xi(q)$ is invertible}
  \State $\theta \leftarrow  \theta - \left( [\nabla \xi(\tilde{q} + M^{-1} \nabla \xi(q) \theta)]^T M^{-1} \nabla \xi(q) \right)^{-1} \xi\left(\tilde{q} + M^{-1} \nabla \xi(q) \theta\right)$
  \Else
  \State $k \leftarrow N_{\rm newt} +1$ and $\eta = 2\varepsilon_{\rm newt}$
  \EndIf
  \State $\eta \leftarrow \max\left(\left|\theta-\theta_{\rm old}\right|,\xi\left(\tilde{q}+M^{-1}\nabla \xi(q)\theta\right)\right)$
  \EndWhile
  \If{$\eta < \varepsilon_{\rm newt}$}
  \State ${\tt Success\_Newton} = {\tt TRUE}$
  \Else
  \State ${\tt Success\_Newton} = {\tt FALSE}$
  \EndIf
  \State \textbf{return} (${\tt Success\_Newton}$, $\theta$)
  \EndProcedure
\end{algorithmic}
\end{algorithm}

\subsection{Numerical results}

\paragraph{Description of the system.}
The configuration of the system is $q = (x,y,z) \in \mathbb{R}^3$. The submanifold $\mathcal M=\{q \in \mathbb{R}^3, \, \xi(q) = 0\}$ is a three-dimensional torus, defined with the function
\[
\xi(q) = \left(R - \sqrt{x^2+y^2}\right)^2 + z^2 - r^2,
\]
where $r>0$ and $R>0$ are two parameters such that $r<R$.
All simulations are started with the initial condition $q^0 = (R+r,0,0)$. The Lagrange multipliers $\lambda^{n+1/2}$ in the first part of the RATTLE scheme are found by solving nonlinear problems of the form
\[
\text{find $a \in \R$ such that }\xi\Big(\tilde{q}+a\nabla \xi\left(q\right)\Big) = 0,
\]
using a Newton's algorithm initialized at $a^0 = 0$ and with termination criterion $|a^{n+1}-a^n| |\nabla \xi\left(q\right) | \leq \eta_{\rm newt}$ (since the error we consider here is an absolue error on the positions rather than a relative one on the Lagrange multipliers, we changed the notation $\varepsilon_{\rm newt}$ from Section~\ref{sec:ex2} to $\eta_{\rm newt}$). 


The GHMC scheme we use is obtained by a variant of Algorithm~\ref{algo:GHMC}, namely a Lie--Trotter splitting of the dynamics where the momenta are first updated as
\[
p^{n+1/4} = P(q^n) \left[ \alpha p^n + \sqrt{1-\alpha^2} \, G^n\right] \text{  where  } P(q) = \mathrm{Id} - \frac{\nabla \xi(q) \otimes \nabla \xi(q)}{|\nabla \xi(q)|^2},
\]
and then a moved is proposed using the integrator $\Psi_{\Delta t}^{\rm rev}$, which is finally accepted or rejected according to the Metropolis-Hastings rule. The parameter $\alpha$ is defined as a function of the parameter~$\gamma > 0$ appearing in~\eqref{eq:Langevin} as
\[
\alpha = \exp (-\gamma \Delta t),
\]
which gives a consistent discretization of the projected Ornstein-Uhlenbeck process on the momenta. The potential function defining the target measure~\eqref{eq:nu} is $V(q) = k|q|^2/2$ for some $k \in \mathbb{R}$, unless otherwise mentioned. In the integrator $\Psi_{\Delta t}^{\rm rev}$, we will either use the force $-\nabla V$ in the Hamiltonian dynamics, or the $0$ force (random walk proposal). In the simulation results reported below, we set $R = 1$ and $r = 0.5$. For the Newton's solver, the tolerance for the convergence of the Lagrange multiplier is fixed to $\eta_{\rm newt} = 10^{-12}$, and limit the number of Newton's iterations to~$N_{\rm newt} = 100$. For the reverse projection check, we use $\eta_{\rm rev} = 10^{-12}$ unless otherwise stated. 

\paragraph{Bias in the distributions.}
We first check how important the reverse projection check is to obtain unbiased samples. To this end, we rely on the following parameterization:
\[
(x,y,z) = \Big( (R+r \cos \phi)\cos\theta, (R+r \cos \phi)\sin\theta, r \sin \phi) \Big),
\]
where $\theta,\phi \in [0,2\pi)$. The random variables~$(\theta,\phi)$ are independent under the invariant measure~\eqref{eq:nu}. The spherical symmetry of the problem implies that $\theta$ should be uniformly distributed on~$[0,2\pi)$. When $k=0$, the theoretical distribution of~$\phi$ has density~\cite{DHS13}
\begin{equation}
  \label{eq:ref_dist}
  m(\phi) = \frac{1}{2\pi}\left(1 + \frac{r}{R} \cos \phi \right).
\end{equation}
We present in Figure~\ref{fig:histo} histograms of the distributions of the sampled angles~$\phi$ using the GHMC scheme for a simulation over $N_{\rm iter} = 10^9$ steps, with $\Delta t = 1$ and $\alpha = 0.5$. The interval $[0,2\pi)$ is decomposed into $N_\phi = 100$ bins of equal widths. The full reverse projection check corresponds to $\eta_{\rm rev} = 10^{-12}$; while the partial reverse projection check corresponds to $\eta_{\rm rev} = 100$ (\emph{i.e.} the Newton's algorithm converges for $\Psi_{\Delta t}$ and $\Psi_{\Delta t} \circ \Psi_{\Delta t}$ but we do not check the equality $(\Psi_{\Delta t} \circ \Psi_{\Delta t})(q^n,p^n) = (q^n,p^n)$). Note that a clear bias on the distributions can be observed when the reverse check is only partially enforced. This shows how important it is to verify that the RATTLE part of the dynamics is actually reversible, and not only well posed for the forward and backward moves. Of course, the bias decreases as $\Delta t$ decreases since the number of rejections related to mismatches $(\Psi_{\Delta t} \circ \Psi_{\Delta t})(q^n,p^n) \neq (q^n,p^n)$ decreases (see the following discussion). 

\begin{figure}
  \includegraphics[width=0.48\textwidth]{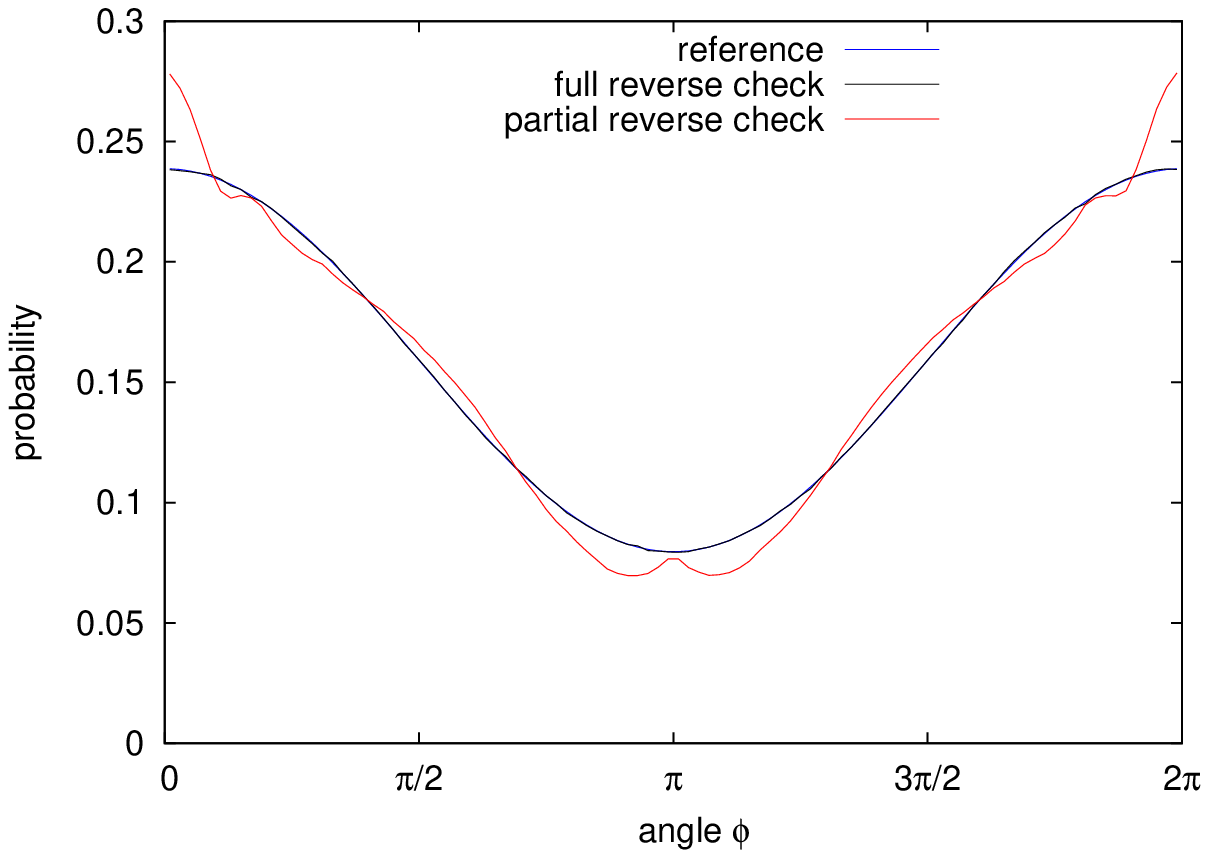}
  \includegraphics[width=0.48\textwidth]{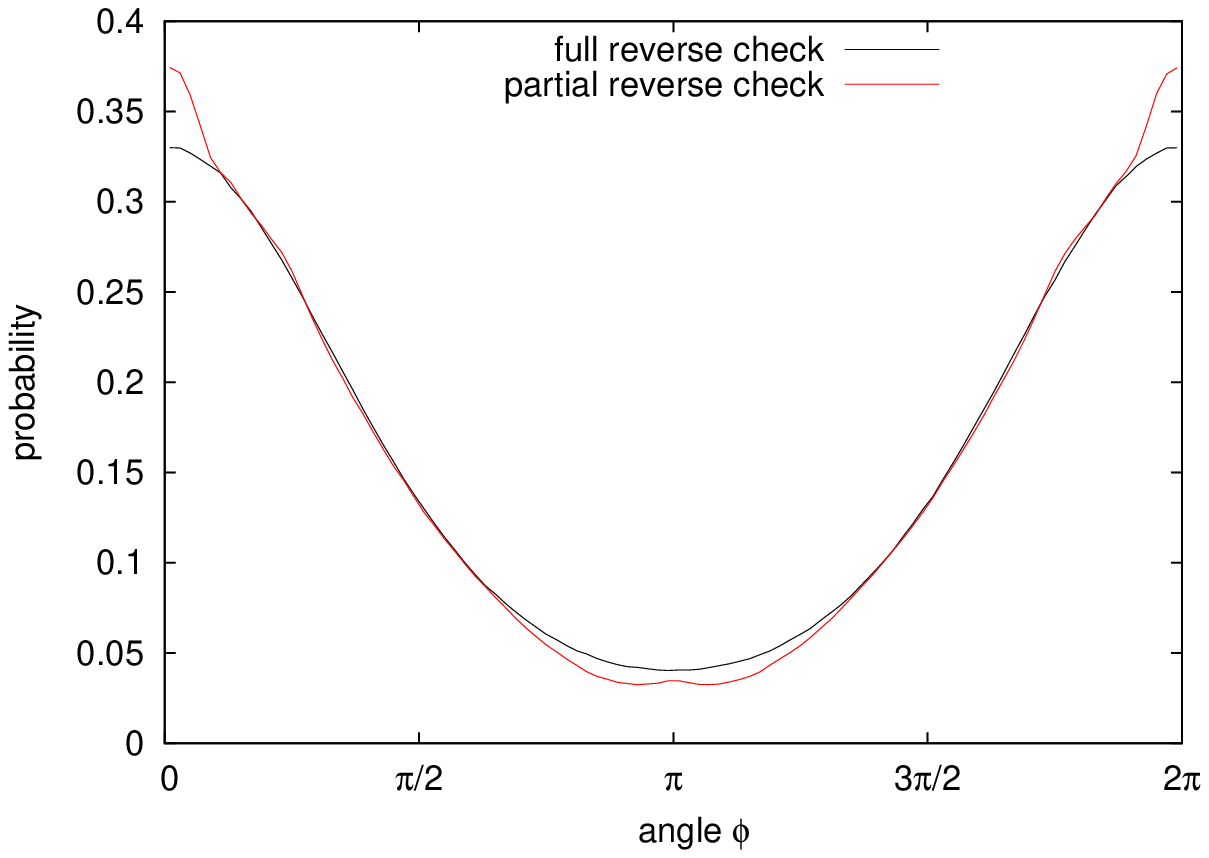}
  \caption{\label{fig:histo} Histograms of the sampled angles~$\phi$ with the GHMC scheme, with full or partial reverse projection check, for $\Delta t = 1$. Left: $k=0$. Right: $k=1$. For $k=0$, one can check that the distribution obtained with the reverse check coincides with the theoretical reference~\eqref{eq:ref_dist}.}
\end{figure}

\paragraph{Rejection rates.}
Let us now study more carefully the rejection rates of the various schemes we consider here, with the constant in the potential set to $k=1$: constrained Metropolis random walk (Algorithm~\ref{algo:HMC} with $k=0$ in the proposal but $k=1$ in the acceptance/rejection criterion; this corresponds to the method proposed in~\cite{zappa-holmes-cerfon-goodman-17}), constrained MALA (Algorithm~\ref{algo:HMC} with $k=1$ in the proposal) and GHMC (Algorithm~\ref{algo:GHMC}). Mean rejections rates are obtained as time averages over trajectories of $N_{\rm iter} = 10^9$ iterations. We distinguish rejections due to (i) the Newton's algorithm not converging for RATTLE starting from a given configuration~$(q^n,p^n)$ (termed ``Newton forward'' in the sequel), (ii) the Newton's algorithm not converging for RATTLE starting from the proposed configuration $\Psi_{\Delta t}(q^n,p^n)$ (called ``Newton reverse''), (iii) the non-reversibility of the RATTLE scheme understood as $q^n \neq Q^n$ with $(Q^n,P^n) = (\Psi_{\Delta t} \circ \Psi_{\Delta t})(q^n,p^n)$ (termed ``non-reversibility'' in the sequel), and (iv) rejections in the Metropolis acceptance/rejection criterion (called ``Metropolis'').

Table~\ref{tab:rejection} presents rejection rates for various values of the parameters. Note first that the total rejection rate increases with the timestep; and that rejection rates for MALA and GHMC are the same, as expected, since the rejection rates correspond to the same average quantities in the RATTLE part of the dynamics under the canonical measure~\eqref{eq:mu}. Most of the rejections are in fact due to the Newton's algorithm not converging for RATTLE starting from $(q^n,p^n)$. On the contrary, once $\Psi_{\Delta t}(q^n,p^n)$ has been obtained, Newton's iterations to apply RATTLE starting from $(\Psi_{\Delta t}\circ\Psi_{\Delta t})(q^n,p^n)$ almost always converge. It is more interesting to consider the rejection rate associated with the non-reversibility criterion: while this criterion is almost always satisfied for small timesteps, it is not for larger timesteps. For MALA and GHMC algorithms with $\Delta t = 1$, about 15\% of moves are rejected because of this condition (see also Figure~\ref{fig:residence}). Comparatively, the rejections due to the Metropolis criterion are less significant. The difference between RW and MALA/GHMC is that the Metropolis rejection rate is always larger with RW, but the rejection rate associated with the non-reversibility condition is smaller for large timesteps. 

\begin{table}
\begin{tabular}{cccccc}
\hline
Method & Total & Newton forward & Newton reverse & non-reversibility & Metropolis \\
\hline
MRW $\dt = 1$ & 0.675 & 0.562 & $3.02 \times 10^{-4}$ & 0.0742 & 0.0385 \\
MALA $\dt = 1$ & 0.675 & 0.509 & $5.83 \times 10^{-4}$ & 0.149 & 0.0167 \\
GHMC $\dt = 1$, $\alpha=0.1$ & 0.675 & 0.509 & $5.83 \times 10^{-4}$ & 0.149 & 0.0167 \\
GHMC $\dt = 1$, $\alpha=0.5$ & 0.675 & 0.509 & $5.83 \times 10^{-4}$ & 0.149 & 0.0167 \\
GHMC $\dt = 1$, $\alpha=0.9$ & 0.675 & 0.509 & $5.83 \times 10^{-4}$ & 0.149 & 0.0167 \\
\hline
MRW $\dt = 0.3$ & 0.158 & 0.0803 & $1.06 \times 10^{-4}$ & 0.0127 & 0.0652 \\
MALA $\dt = 0.3$ & 0.107 & 0.0763 & $1.22 \times 10^{-4}$ & 0.0138 & 0.0168 \\
GHMC $\dt = 0.3$, $\alpha=0.1$ & 0.107 & 0.0763 & $1.22 \times 10^{-4}$ & 0.0138 & 0.0168 \\
GHMC $\dt = 0.3$, $\alpha=0.5$ & 0.107 & 0.0763 & $1.22 \times 10^{-4}$ & 0.0138 & 0.0168 \\
GHMC $\dt = 0.3$, $\alpha=0.9$ & 0.107 & 0.0763 & $1.22 \times 10^{-4}$ & 0.0138 & 0.0168 \\
\hline
MRW $\dt = 0.1$ & 0.0259 & $5 \times 10^{-7}$ & 0 & $7 \times 10^{-8}$ &  0.0259 \\
MALA $\dt = 0.1$ & $6.73 \times 10^{-4}$ & $5 \times 10^{-7}$ & $10^{-9}$ & $5 \times 10^{-8}$ & $6.73 \times 10^{-4}$ \\
GHMC $\dt = 0.1$, $\alpha=0.1$ & $6.72 \times 10^{-4}$ & $5 \times 10^{-7}$ & $10^{-9}$ & $6 \times 10^{-8}$ & $6.72 \times 10^{-4}$ \\
GHMC $\dt = 0.1$, $\alpha=0.5$ & $6.73 \times 10^{-4}$ & $5 \times 10^{-7}$ & $2 \times 10^{-9}$ & $8 \times 10^{-8}$ & $6.72 \times 10^{-4}$ \\
GHMC $\dt = 0.1$, $\alpha=0.9$ & $6.74 \times 10^{-4}$ & $5 \times 10^{-7}$ & 0 & $7 \times 10^{-8}$ & $6.73 \times 10^{-4}$ \\
\hline
\end{tabular}
\caption{\label{tab:rejection} MRW: Metropolis random walk. Total: total rejection rate. Newton forward: rejection rate corresponding to failed convergence of Newton in the RATTLE proposal (see Step~1 in Section~\ref{sec:psi_B}). Newton reverse: rejection rate corresponding to failed convergence of Newton in the reversal check (see Step~3 in Section~\ref{sec:psi_B}). Non-reversibility: rejection rate corresponding to the fact that the initial position is different from the one obtained by the reverse move (see Step 5 in Section~\ref{sec:psi_B}). Metropolis: rejection rate arising from the Metropolis acceptance/rejection step.}
\end{table} 

\paragraph{Relevance of the full reverse projection check.} Let us finally demonstrate that the reverse projection check is indeed useful in practice, by first determining the optimal timestep in some sense, and then checking that the rejection rate due to $(q^n,p^n) \neq (\Psi_{\Delta t} \circ \Psi_{\Delta t})(q^n,p^n)$ is not small for this timestep. To this end, we consider a double well potential
\[
V(x,y,z) = k\left(x^2-R^2\right)^2,
\] 
again on the torus introduced above. Note that the target measure~\eqref{eq:mu} then concentrates around the portions of the torus corresponding to~$x=-R$ and $x=R$. Typical realizations of trajectories obtained with the GHMC scheme for $k=5$ and $\alpha = 0.5$ are presented in Figure~\ref{fig:traj}. One observes indeed two metastable states, respectively around $x=-R$ and $x=R$. For large timesteps, one observes long durations over which the trajectory remains stuck at the same position, which corresponds to successive rejections of the proposed moves.

\begin{figure}
  \includegraphics[width=0.48\textwidth]{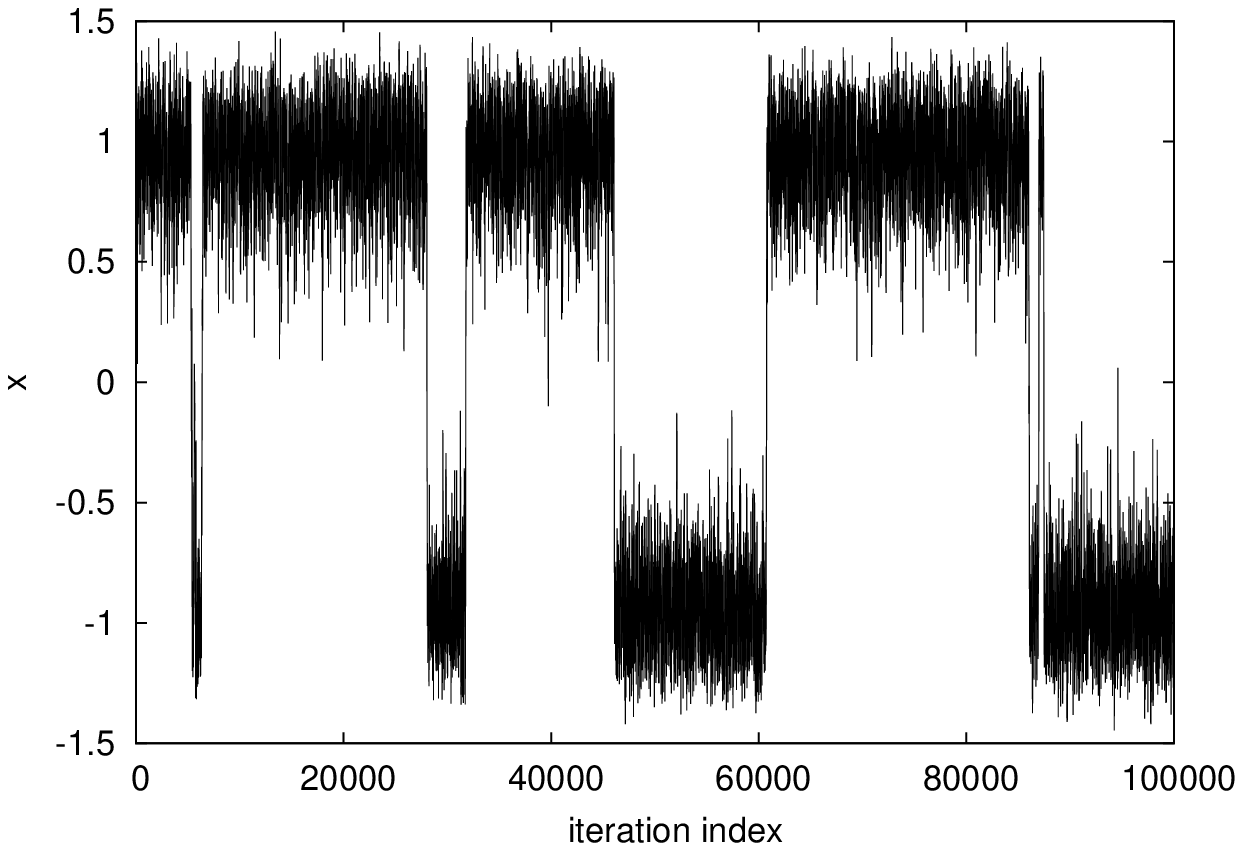}
  \includegraphics[width=0.48\textwidth]{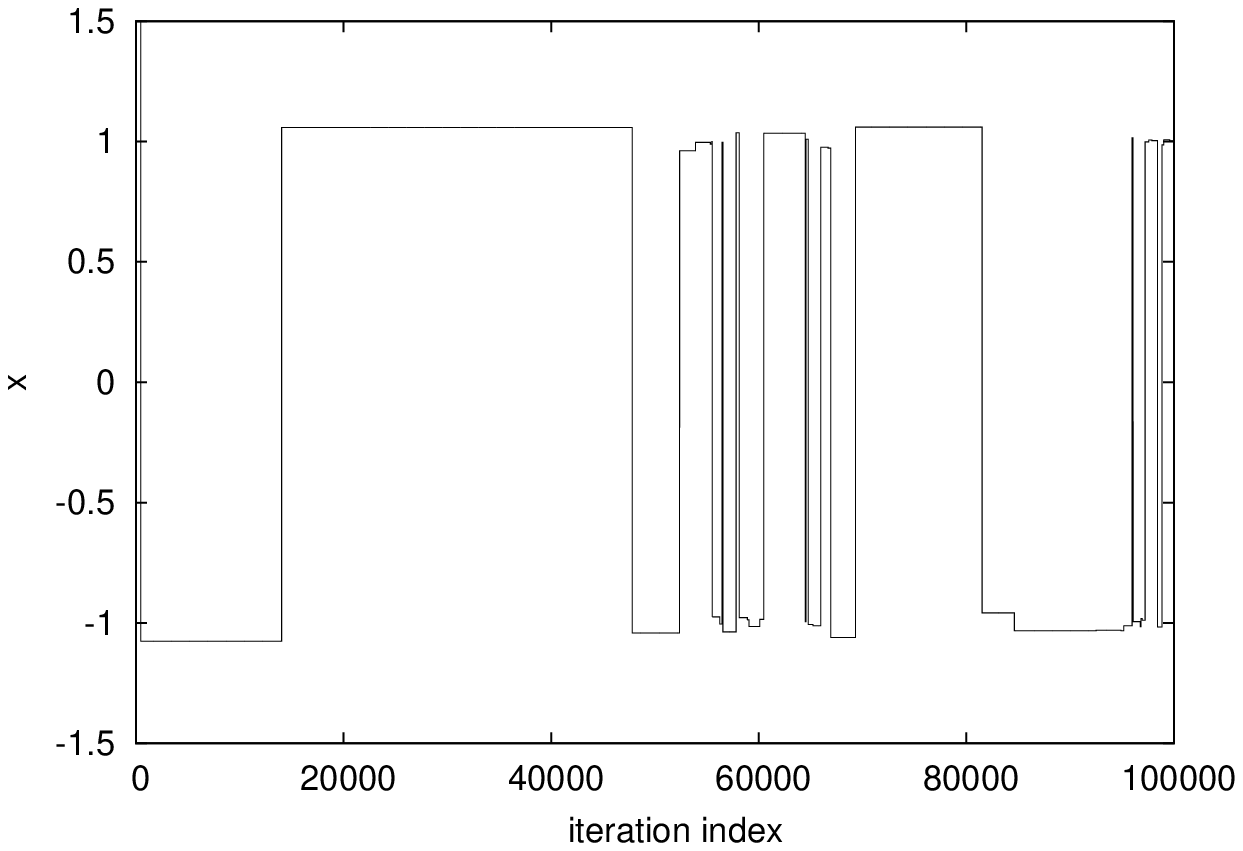}
  \caption{\label{fig:traj} Typical trajectories of the GHMC dynamics for the double well potential. Left: $\Delta t = 0.05$. Right: $\Delta t = 3$.}
\end{figure}

For any such trajectory, we compute the mean residence duration in metastable states, namely the average number of steps it takes to go from the vicinity of one minimum of~$V$ to the other (this does not depend on the starting minimum by the symmetry of the problem). Let us emphasize that we count the number of iterations, and not the corresponding physical time, since the timestep is seen here as some adjustable parameter. More precisely, we consider the initial condition $q^0 = (R+r,0,0)$ as well as the additional variable $\Theta^0 = 1$. This additional variable keeps track of the metastable state in which the dynamics is: it is equal to~1 when the sampled positions are close to the minimum around~$R$ and equal to~$-1$ when~$x^n$ is around~$-R$. We define the following stopping times and values of the additional variable: $\tau^0 = 0$, and 
\[
\tau^{k+1} = \min \{ n > \tau^k \, | \, \Theta^k x^n < -R \}, \qquad \Theta^{k+1} = -\Theta^k.
\]
Denoting by~$K$ the (random) number of successful switches for $N_{\rm iter}$ iteration steps, the mean residence duration is estimated as 
\[
\widehat{\tau} = \frac{1}{K} \sum_{k=1}^K \tau^k.
\]
We present in Figures~\ref{fig:residence} and~\ref{fig:residence_V0} the mean residence duration as a function of the timestep, for HMC and GHMC for two values of $\gamma$, using either $\nabla V$ in the RATTLE scheme or setting it to~0. HMC in the latter case then reduces to a Metropolis random walk. Note that, for small timesteps, the mean residence duration scales as $1/\Delta t$ for GHMC (which is indeed a discretization of the Langevin dynamics with a timestep~$\Delta t$), while it scales as $1/\Delta t^2$ for MALA (which is indeed a discretization of the constrained overdamped Langevin dynamics with timestep~$\Delta t^2/2$, see the discussion at the beginning of Section~\ref{sec:constrained_HMC}).

The results show that the optimal timestep, namely the one which minimizes the mean residence duration, is of the order of~0.7 when $\nabla V$ is used in the RATTLE proposal. For such timesteps, the rejections due to the non reversibility condition $(q^n,p^n) \neq (\Psi_{\Delta t} \circ \Psi_{\Delta t})(q^n,p^n)$ are of the order of 15-20\%, the total rejection rate being about 90\%. When forces are set to~0 in the RATTLE proposal, the optimal timestep is of the order of~1, the rejections due to the non reversibility condition are about 5\%, for a total rejection rate of 80-85\%. Note that, in this specific low dimensional example, random-walk type proposals (obtained by setting forces to~0 in the RATTLE proposal) are slightly more efficient for large timesteps in terms of mean residence duration. However, including gradient forces in the proposal is known to be important in some high dimensional cases, see for example~\cite{RR98}. In any case, at least on this simple example, all the algorithms are the most efficient for timesteps which indeed require a full reverse projection check.

\begin{figure}
  \includegraphics[width=0.48\textwidth]{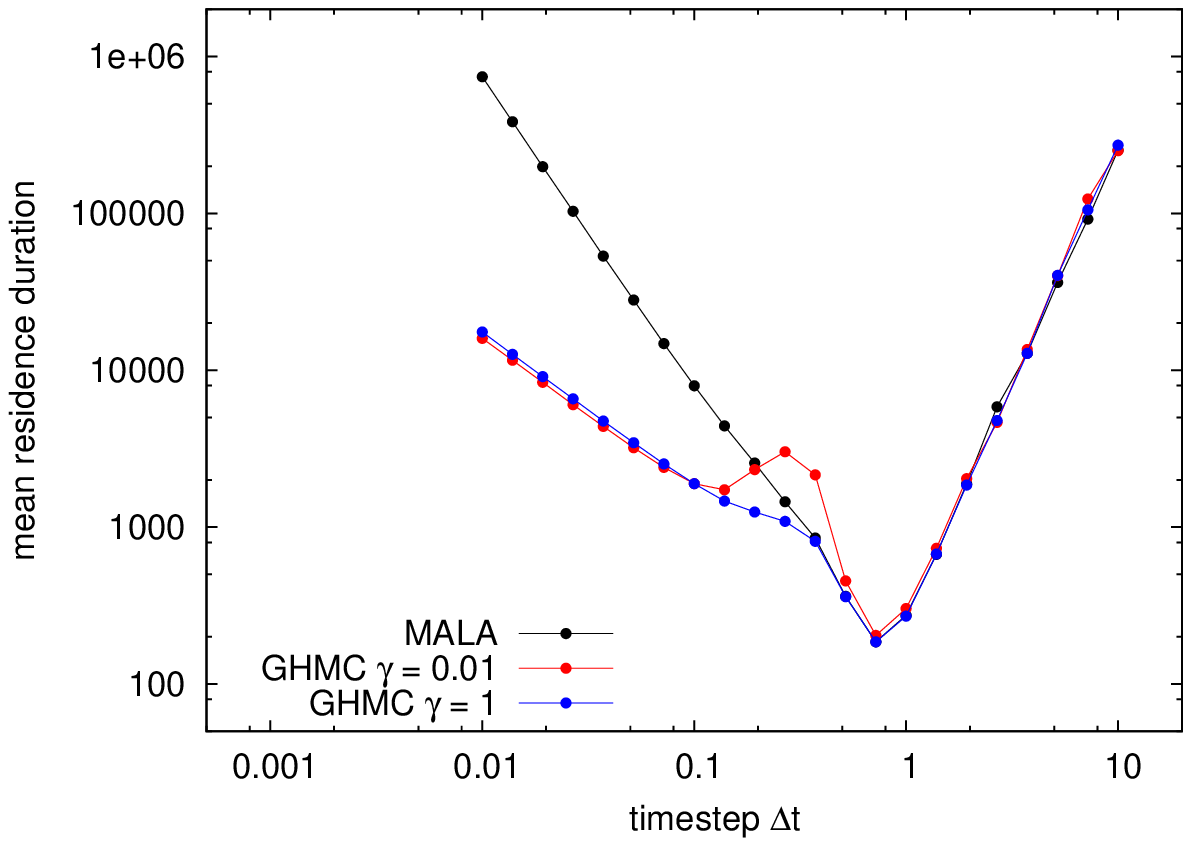}
  \includegraphics[width=0.48\textwidth]{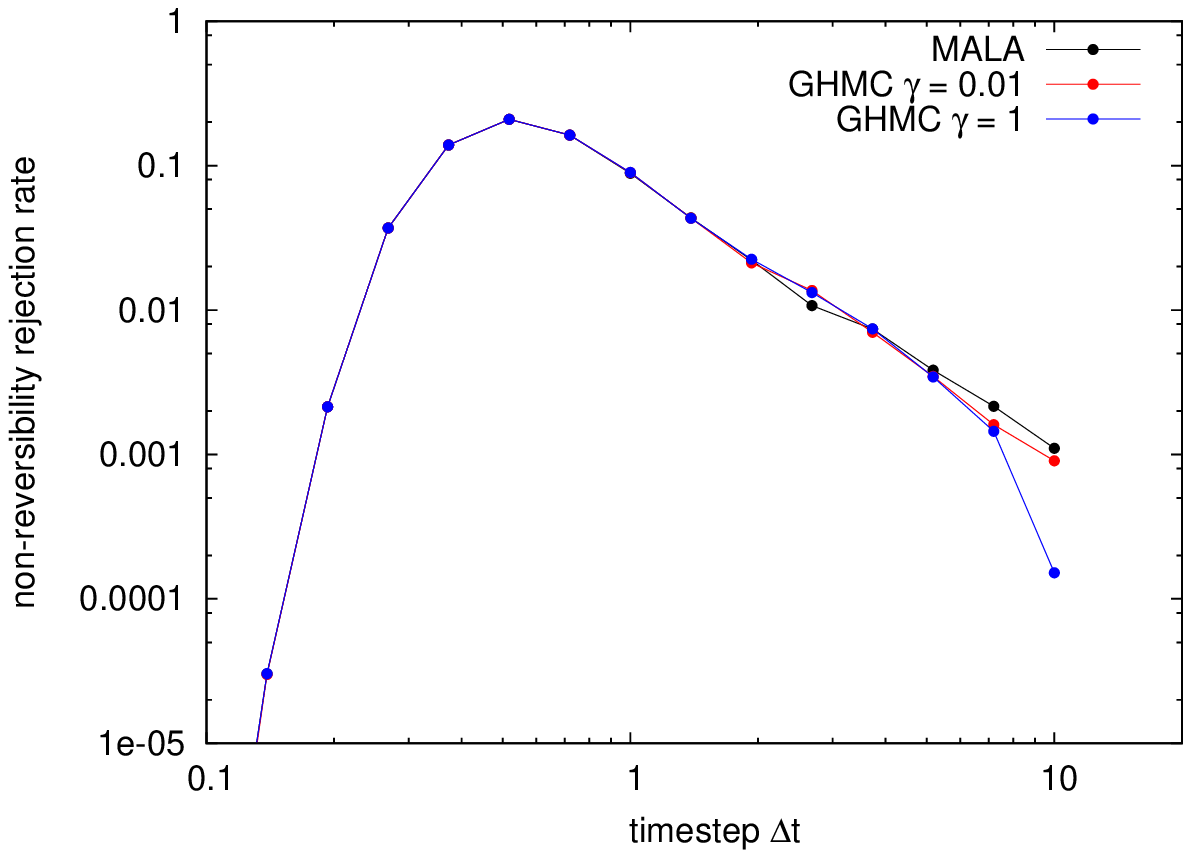}
  \caption{\label{fig:residence} Left: mean residence duration as a function of the timestep. Right: non-reversibility rejection rate.}
\end{figure}

\begin{figure}
  \includegraphics[width=0.48\textwidth]{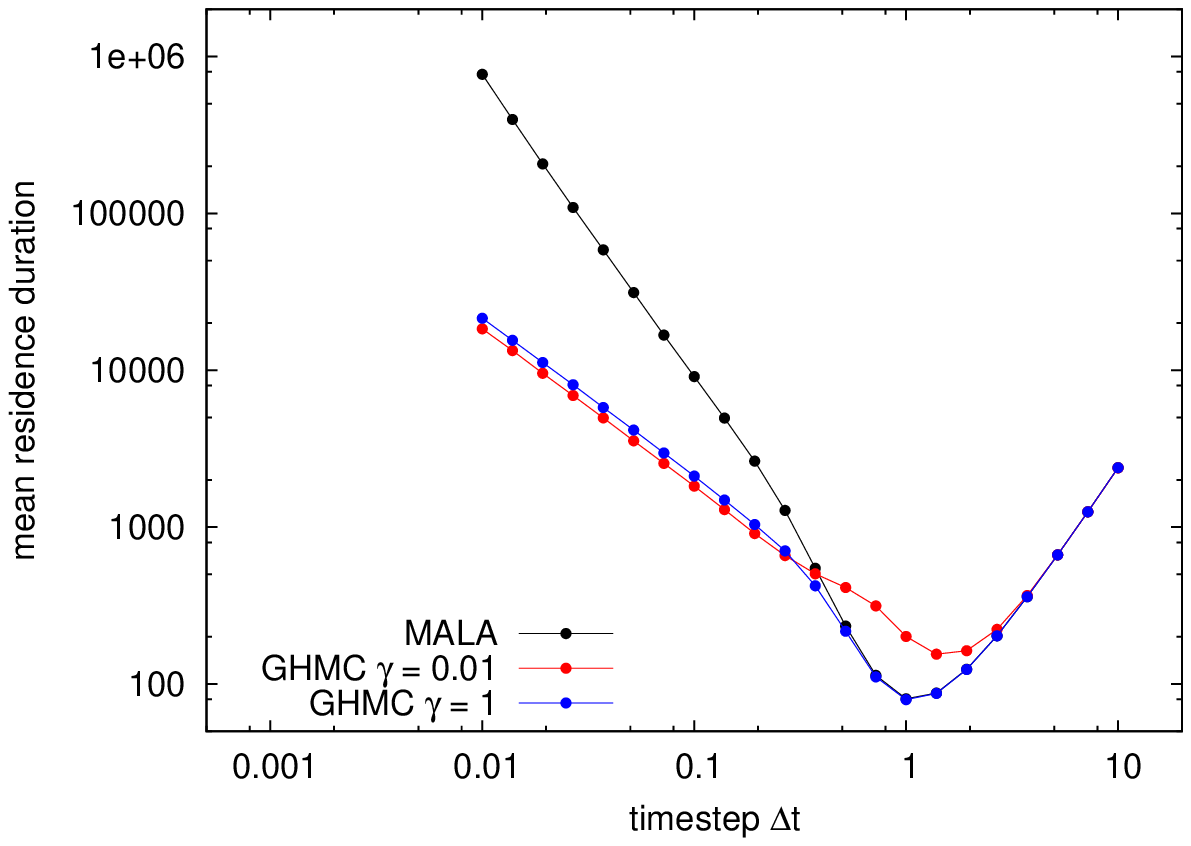}
  \includegraphics[width=0.48\textwidth]{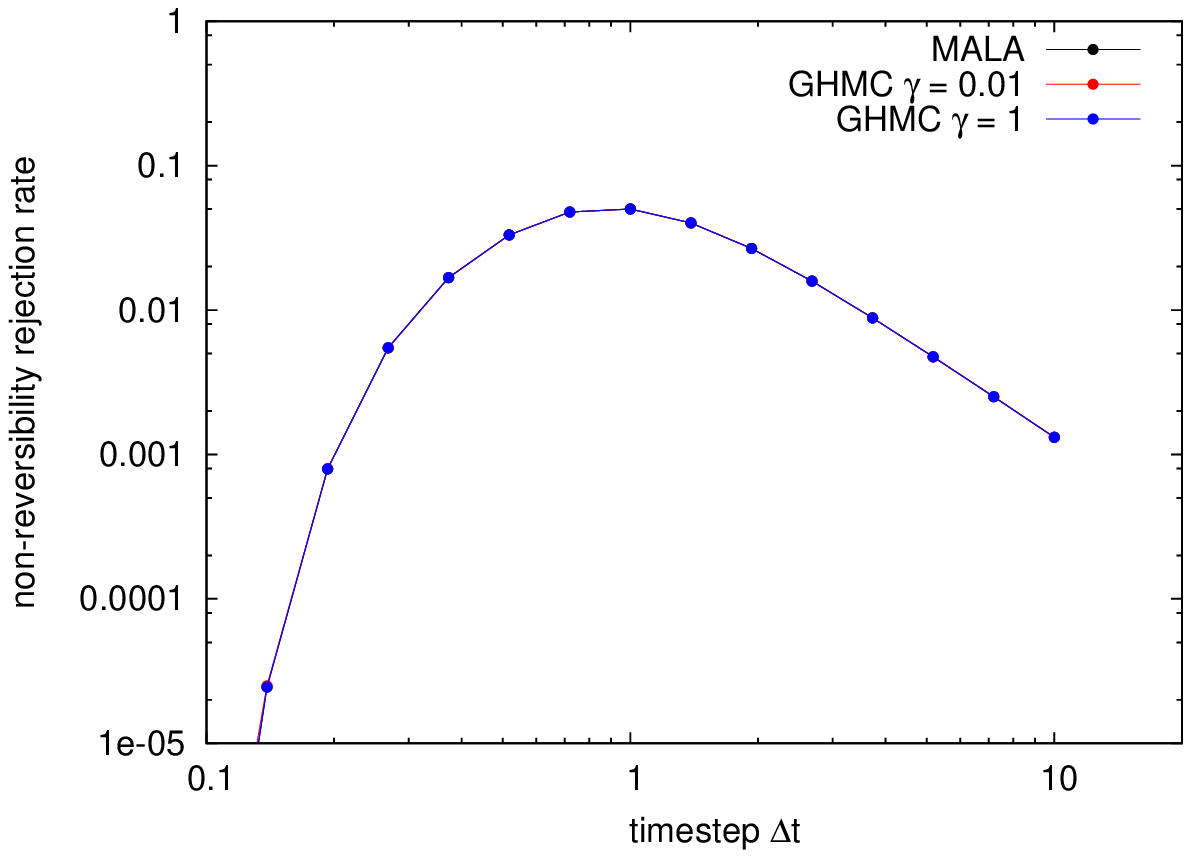}
  \caption{\label{fig:residence_V0} Same as in Figure~\ref{fig:residence} except that the forces are set to~0 in the RATTLE, for the proposal moves.}
\end{figure}

\paragraph{Acknowledgments.}
We thank Christian Robert for pointing out the reference~\cite{zappa-holmes-cerfon-goodman-17} as soon as it was preprinted on arXiv, as well as Jonathan Goodman and Miranda Cerfon--Holmes for very useful discussions. We also thank Shiva Darshan and Miranda Holmes--Cerfon for pointing out a typo in the Numerical Algorithm~\ref{algo:global} in this manuscript. This work  was funded in part by the European Research Council under the European Union's Seventh Framework Programme (FP/2007-2013) / ERC Grant Agreement number 614492, as well as the Agence Nationale de la Recherche, under grant ANR-14-CE23-0012 (COSMOS). We also benefited from the scientific environment of the Laboratoire International Associ\'e between the Centre National de la Recherche Scientifique and the University of Illinois at Urbana-Champaign. 


\end{document}